\RequirePackage{fix-cm}
\documentclass[smallextended,numbook,envcountsame]{svjour3}
\smartqed  
\usepackage{graphicx}
\usepackage{amsmath}
\usepackage{amssymb}

\font \mymathbb = bbold10 at 11pt

\newcommand{\one}{\mbox{\mymathbb{1}}}

\newcommand{\beq}{\begin{equation}}
\newcommand{\eeq}{\end{equation}}

\usepackage[numbers,comma,square,sort&compress]{natbib}
\usepackage{color,colordvi}
\definecolor{my-blue}{rgb}{0.0,0.0,0.6}
\definecolor{my-red}{rgb}{0.5,0.0,0.0}
\definecolor{my-green}{rgb}{0.0,0.5,0.0}
\RequirePackage[colorlinks,urlcolor=my-blue,linkcolor=my-red,citecolor=my-green]{hyperref}
\RequirePackage{hypernat}

\usepackage{wrapfig}

\newcommand\pell{\hat p_\ell} 
\newcommand\ratell{H_{\ell}}  
\newcommand\bigom{\mathbf \Omega} 
\newcommand\range{{\mathcal R}} 
\newcommand\wz{\eta} 
\newcommand\MC{\cQ} 
\newcommand\measures{\cM_1} 
\newcommand\Sopr{S^+}

\newcommand\Ll{\cL} 
\newcommand\Uset{\cU} 
\newcommand\unif{\kU_b} 
\newcommand\pres{\Lambda}  
\newcommand\oneell{\ell} 
\newcommand\gr{{\mathcal G}}

\newcommand*{\Z}{{\mathbb Z}}
\newcommand*{\kS}{{\mathfrak S}}

\newcommand*{\cP}{{\mathcal P}}
\newcommand*{\w}{\omega}

\newcommand*{\E}{{\mathbb E}}

\newcommand*{\cM}{{\mathcal M}}
\newcommand*{\N}{{\mathbb N}}
\newcommand*{\R}{{\mathbb R}}
\newcommand*{\cQ}{{\mathcal Q}}
\providecommand{\abs}[1]{\left\vert#1\right\vert}

\newcommand*{\cU}{{\mathcal U}}
\newcommand*{\kU}{{\mathfrak U}}

\newcommand*{\Q}{{\mathbb Q}}
\renewcommand*{\P}{{\mathbb P}} 
\newcommand*{\xhat}{{\hat x}}
 \newcommand*{\gbar}{{\bar g}}

\newcommand*{\cL}{{\mathcal L}}

\newcommand*{\e}{\varepsilon}

\newcommand*{\zhat}{{\hat z}}

\newcommand*{\bfu}{{\mathbf u}}
\newcommand*{\fl}[1]{\left\lfloor{#1}\right\rfloor}
\newcommand*{\ce}[1]{\left\lceil{#1}\right\rceil}
\newcommand*{\xtil}{{\tilde x}}

\newcommand*{\uhat}{{\hat u}}

\newcommand*{\nn}{\nonumber}

\newcommand*\cA{{\mathcal A}}
\newcommand*\bfv{{\mathbf v}}

\newcommand\ri{\mathrm{ri}\,\,}
\newcommand\conv{\mathrm{co}\,} 
\newcommand\kernel{\mathrm{ker}\,} 
\newcommand\aff{\mathrm{aff}\,} 
\newcommand\ex{\mathrm{ex}\,}

 \newcommand\cH{\mathcal H}

\def\note#1{\textup{\textsf{\color{blue}{\bf(((} #1 {\bf)))}}}}

\begin{document}

\title{Quenched Point-to-Point Free Energy\\ for Random Walks in Random Potentials
}

\titlerunning{Quenched Point-to-Point Free Energy}        

\author{Firas Rassoul-Agha         \and
        Timo Sepp\"al\"ainen 
}


\institute{F.\ Rassoul-Agha \at
               Mathematics Department, University of Utah, 155 South 1400 East, Salt Lake City, UT 84109, USA\\
              \email{firas@math.utah.edu}           
           \and
           T.\ Sepp\"al\"ainen \at
              Mathematics Department, University of Wisconsin-Madison, 419 Van Vleck Hall, Madison, WI 53706, USA\\
              \email{seppalai@math.wisc.edu}
}

\date{Received: February 2012 / Accepted: date}

\maketitle

\begin{abstract}
We consider a  random walk in a random potential on a square
lattice of arbitrary dimension.    The potential is a function of an ergodic
environment and   steps of the walk. The potential   is  subject to a 
moment assumption whose strictness is tied to the mixing of the environment, the
  best case   being the 
i.i.d.\ environment.   
 We prove that the infinite volume quenched point-to-point free energy exists and has a variational formula in terms of  entropy.
We  establish regularity properties of the point-to-point  free energy,  
 and link it  to the infinite volume point-to-line free energy and quenched  large deviations of the  walk.  One corollary is a  quenched  large deviation principle 
for random walk in an ergodic random environment,  with a continuous rate function.  
  
\end{abstract}


\section{Introduction}

This paper studies the limiting free energy of a class of models 
 with  Boltzmann-Gibbs-type distributions on random walk paths. 
The 
energy of a   path  is defined through a coupling of the walk
 with a random environment. 
 Our main interest is   the {\sl directed polymer   in an i.i.d.\ random
environment}, also called the {\sl polymer with  bulk disorder}.  
This model was  
introduced in the statistical 
physics literature by Huse and Henley in 1985 \cite{Hus-Hen-85}.  
For recent surveys see
\cite{Com-Shi-Yos-04, Hol-09}.  
The free energy of these models is a central object of study. 
Its dependence on model parameters gives information about
phase transitions.  In quenched settings the fluctuations  
of the quenched free energy are closely related to the fluctuations of the
path.

Some   properties we develop
can   be proved with  little or no extra cost  more generally.
   The  formulation then consists of a general walk  in a 
potential that can  depend both on an ergodic  environment and on the steps
of the walk.  We call the model random walk in a random potential (RWRP).  

This paper concentrates mainly on   the point-to-point version of   RWRP 
 where the
walk is fixed at two points and allowed to fluctuate in between. 
The point-to-line model was studied in the companion paper
\cite{Ras-Sep-Yil-12-}.   The  motivation for both papers  was that the free energy  
was  known  only as a subadditive limit, with no explicit
formulas.  
We provide two 
variational formulas for the point-to-point free energy. One comes in terms of entropy and 
 we develop it in detail after   preliminary work on 
the    regularity of the free energy.   The other involves
correctors (gradients of sorts) and can be deduced by  
combining a convex duality given in \eqref{pres19} below with
 Theorem 2.3 from \cite{Ras-Sep-Yil-12-}.   
 
Significant  recent progress  has taken place in the realm of 1+1 dimensional
exactly solvable directed polymers (see review \cite{Cor-11-}).  
Work on general models is far behind.   
Here are three future directions    opened up by our  
  results in the  present   work and \cite{Ras-Sep-Yil-12-}.  
  
 (i)   One    goal is to use this theory to access
properties of the limiting  free energy, especially in regimes of strong
disorder where the quenched model and annealed model deviate from each
other.

 (ii) The variational formulas
identify  certain natural corrector functions 
and Markov processes whose investigation should shed light on the
polymer models themselves. Understanding this picture 
for   the exactly solvable log-gamma polymer  \cite{Sep-12-} will be the 
first step.  

(iii)  The zero-temperature limits of polymer models are last-passage
percolation models.  In this limit the free energy turns into the limit shape.
Obtaining  information about limit shapes of percolation models
 has been notoriously 
difficult.  A future direction is to extend the variational formulas 
to the zero-temperature case. 

In the remainder of the introduction we 
describe the model and some examples, 
 give   an overview of the paper, and describe some past literature. 

\medskip

{\bf The RWRP model and examples.}   
 Fix a dimension $d\in\N$. 
Let  $\range\subset\Z^d$ be a finite subset of the square lattice and
let $P$ denote the distribution of the {\sl random walk} on $\Z^d$ started at $0$ and whose transition probability is 
$\hat p_z=1/|\range|$ for $z\in\range$ and $\hat p_z=0$ otherwise.
In other words, the random walk picks its steps uniformly at random from $\range$. 
 $E$   denotes expectation under $P$.   
 $\range$ generates the  
additive  group $\gr= \{\sum_{z\in\range}a_z z:a_z\in\Z\}$. 

An {\sl environment} $\w$ is a sample point from a probability space $(\Omega, \kS, \P)$.
$\Omega$ comes  equipped with a group  $\{T_z:{z\in\gr}\}$   of 
measurable commuting transformations that satisfy 
$T_{x+y}=T_xT_y$ and $T_0$ is the identity.
  $\P$ is a $\{T_z:z\in\gr\}$-invariant probability measure on $(\Omega,\kS)$. 
  This is summarized by the statement that $(\Omega,\kS,\P,\{T_z:z\in\gr\})$Ê
is a measurable dynamical system.   
  As usual 
  $\P$ is {\sl ergodic}  
 if  $T_z^{-1}A=A$ for all $z\in\range$ implies   $\P(A)=0$ or $1$, 
for events $A\in\kS$.
A stronger assumption of {\sl total ergodicity} 
  says that   $\P(A)=0$ or $1$  whenever 
$T_z^{-1}A=A$ for some extreme point $z$ of the convex hull of $\range$.
   $\E$ will denote expectation relative to $\P$.

A  {\sl potential}  is  
  a measurable function $g:\Omega\times\range^\ell\to\R$ for
some integer $\ell\ge0$. The case $\ell=0$ means that 
$g=g(\w)$, a function of $\w$ alone. 

 Given an environment $\w$ and an integer $n\ge1$
define the {\sl quenched polymer measure} 
\beq Q^{g,\w}_{n}(A) 
	=\frac1{Z_{n}^{g,\w}}E\bigl[e^{\sum_{k=0}^{n-1}g(T_{X_k}\w, \,Z_{k+1,k+\ell})}\one_A(\w, X_{0,\infty})\bigr],
	\label{Q^g}\eeq 
where $A$ is an event on environments and paths and 
 	\[Z_{n}^{g,\w}=E\big[e^{\sum_{k=0}^{n-1}g(T_{X_k}\w,\, Z_{k+1,k+\ell})}\big]\]
is the  normalizing constant called the {\sl quenched partition function}.
This model we call {\sl random walk in a random potential} (RWRP).   Above
 $Z_k=X_k-X_{k-1}$ is a random walk step and $Z_{i,j}=(Z_i,\dotsc,Z_j)$ a vector of steps.  
 Similar notation  will be used for all finite and infinite vectors and path segments, including 
 $X_{k,\infty}=(X_k, X_{k+1}, \dotsc)$ and $z_{1,\ell}=(z_1,\dotsc, z_\ell)$ used above. 
 Note that in general the measures $Q^{g,\w}_{n}$
defined in \eqref{Q^g} are not consistent as $n$ varies.  
Here are some key examples of the setting.  

\begin{example}[I.I.D.\ environment.]\label{ex:product}
A natural   setting  is the one where $\Omega=\Gamma^{\Z^d}$
is  a product space with generic points $\w=(\w_x)_{x\in\Z^d}$ and  translations $(T_x\w)_y=\w_{x+y}$,
the coordinates $\w_x$ are i.i.d.\ under $\P$,   
  and $g(\w,z_{1,\ell})$ a local function of $\w$, which means that 
  $g$ depends on only finitely many coordinates $\w_x$.  This is a totally ergodic case. 
 In this setting $g$ has the {\sl $r_0$-separated i.i.d.\ property}
 for some positive integer $r_0$. 
By this we mean that  if $x_1,\dotsc, x_m\in\gr$   satisfy
$\abs{x_i-x_j}\ge r_0$ for $i\ne j$, then 
  the  $\R^{\range^\ell}$-valued random vectors    
$\{ \bigl(g(T_{x_i}\w, z_{1,\ell})\bigr)_{z_{1,\ell}\in\range^\ell}: 1\le i\le m\}$ 
are i.i.d.\ under $\P$.  
     \end{example}

\begin{example}[Strictly directed walk and local potential in  i.i.d.\ environment.] 
A specialization of Example \ref{ex:product} where  
 $0$ lies  outside the convex hull of $\range$.
This is equivalent to the existence of  $\uhat\in\Z^d$ such that $\uhat\cdot z>0$ for all $z\in\range$.
\label{ex:dir-iid} \end{example}

\begin{example}[Stretched polymer.]  A stretched polymer has an external field $h\in\R^d$
that biases the walk, so the potential is 
$g(\w,z)=\Psi(\w)+h\cdot z$.  See the survey paper \cite{Iof-Vel-12-} and its references   for the state of the art on stretched polymers
in a product potential.
\label{ex:stretch}\end{example}

\begin{example}[Random walk in random environment.]\label{ex:rwre}
To cover RWRE   take $\ell=1$ and $g(\w,z)=\log p_z(\w)$ where
$(p_z)_{z\in\range}$ is a measurable mapping from $\Omega$ into $\cP=\{(\rho_z)_{z\in\range} \in[0,1]^\range:\sum_z \rho_z=1\}$, the space of
probability distributions on $\range$.  The 
{\sl quenched path measure}  $Q^\w_0$ of RWRE
started at $0$ is the probability measure on the path space $(\Z^d)^{\Z_+}$ 
  defined by the initial condition $Q^\w_0(X_0=0)=1$ and 
the transition probability   $Q^\w_0(X_{n+1}=y\vert X_n=x)
=p_{y-x}(T_x\w)$.    The 
$(X_0,\dotsc,X_n)$-marginal of the polymer measure   $Q^{g,\w}_n$ 
in \eqref{Q^g} 
  is the marginal of the quenched path measure   $Q^\w_0$.
\end{example}

{\bf Overview of the paper.} 
 Under some assumptions article  \cite{Ras-Sep-Yil-12-}  proved   the 
$\P$-almost sure existence of the limit
\beq   \pres_\ell(g) = \lim_{n\to\infty} n^{-1} \log E\big[e^{\sum_{k=0}^{n-1}g(T_{X_k}\w,Z_{k+1,k+\ell})}\big] . \label{pres-8.4} \eeq
In different contexts this is called 
  the  {\sl limiting logarithmic moment generating function},  the {\sl pressure},
or  the {\sl free energy}.   One of the main results of \cite{Ras-Sep-Yil-12-}
was the variational characterization 
\beq  \pres_\ell(g)=\sup_{\substack{\mu\in\measures(\bigom_\ell), c>0}}
\bigl\{E^\mu[\min(g,c)]-\ratell(\mu)\bigr\}.   
\label{pr-ent-4}
\eeq
$\measures(\bigom_\ell)$ is the space of probability measures on $\bigom_\ell=\Omega\times\range^\ell$ 
and $\ratell(\mu)$ is an entropy, defined in \eqref{Helldef} below.  

 In the present 
paper  we study  the 
quenched {\sl point-to-point free energy}   
\beq \pres_\ell(g,\zeta)=\lim_{n\to\infty} n^{-1} \log E\big[e^{\sum_{k=0}^{n-1}g(T_{X_k}\w,Z_{k+1,k+\ell})}\one\{X_n= \xhat_n(\zeta)\}\big]\label{pres-8.6}\eeq
where $\zeta\in\R^d$ and  $\xhat_n(\zeta)$ is a lattice point that approximates $n\zeta$.
 Our main result is a variational characterization of $\pres_\ell(g,\zeta)$
 which is identical  
 to \eqref{pr-ent-4}, except that now   the supremum is over distributions $\mu$ on $\bigom_\ell$
 whose mean velocity for the path is $\zeta$.    For
  directed walks in i.i.d. environments this   is   Theorem \ref{th:pr-ent2}
 in Section \ref{iid-sec}.

We begin  in Section \ref{sec:free energy} with the existence of 
$\pres_\ell(g,\zeta)$  and  
  regularity  in  $\zeta$.  
A by-product is   an  independent proof of the limit \eqref{pres-8.4}.
We  relate
$ \pres_\ell(g)$ and $\pres_\ell(g,\zeta)$ to each other in a couple different ways. 
This relationship yields a second  variational formula for $\pres_\ell(g,\zeta)$. 
Combining  convex duality \eqref{pres19}  with
 Theorem 2.3 from \cite{Ras-Sep-Yil-12-} gives 
  a   variational formula for $\pres_\ell(g,\zeta)$
that involves tilts and  corrector functions rather than measures.  
 
  Section \ref{sec:cont}   proves further regularity properties for the 
i.i.d.\ strictly directed case:   continuity of $\pres_\ell(g,\zeta)$   in $\zeta$  and  $L^p$ continuity ($p>d$)     in $g$. 

  Section \ref{ld-sec} is for large deviations.  
Limits \eqref{pres-8.4} and \eqref{pres-8.6} give  a 
 quenched  large deviation principle for the distributions 
  $Q^{g,\w}_{n}\{X_n/n\in \cdot\,\}$, with  rate function  $I^g(\zeta) =  \pres_\ell(g)-\pres^{\mathrm{usc}(\zeta)}_\ell(g,\zeta)$
where  $\pres^{\mathrm{usc}(\zeta)}_\ell(g,\zeta)$ is the
upper semicontinuous regularization.  This rate function is continuous 
on the convex hull of $\range$.  
  We specialize the LDP    to
  RWRE  and give an overview of past work on quenched large deviations for   RWRE.  
 
 Section  \ref{iid-sec} develops the entropy representation   
 of $\pres_\ell(g,\zeta)$  for the i.i.d.\ strictly directed
 case.   The general case can be found in the preprint version \cite{Ras-Sep-12-arxiv}. 
   The LDP is the key, 
  through a contraction principle.  
 
Our  results are valid for   unbounded potentials,  provided we have   
  control of the mixing of the environment.  
   When shifts of the potential are strongly mixing,   $g\in L^p$ 
 for $p$ large enough suffices.  In particular, for an i.i.d.\ environment and stricly directed walks,  the assumption is
 that $g$ is local in its dependence on $\w$ and $g(\cdot\,,z_{1,\ell})\in L^p(\P)$ for
 some $p>d$.  

  Section \ref{L2-sec} illustrates the theory 
for     a directed polymer in an i.i.d.\ environment
in the $L^2$ region (weak disorder, dimension $d\ge 3$). 
The variational formula is solved by an RWRE  in a correlated environment, 
and a   tilt (or   ``stretch'' as in Example 
\ref{ex:stretch}) appears as the dual variable of the velocity $\zeta$.

 \medskip
 
{\bf Literature and past results.}  Standard  references for   RWRE are
\cite{Bol-Szn-02-dmv},  \cite{Szn-04} and  \cite{Zei-04},  
 and for RWRP   \cite{Com-Shi-Yos-04}, 
\cite{Hol-09} and  
\cite{Szn-98}. 
  RWRE large deviations literature is recounted
 in Section \ref{ld-sec}
after Theorem \ref{rwre-ldp|}.   Early forms of our variational formulas 
appeared in  position-level large deviations for RWRE
in \cite{Ros-06}. 

A notion related to the 
 free energy is 
the  {\sl Lyapunov exponent}    defined by 
\[\lim_{n\to\infty} n^{-1}\log E\Big[e^{\sum_{k=0}^{\tau(\xhat_n(\zeta))-1}g(T_{X_k}\w,Z_{k+1,k+\ell})}\one\{\tau(\xhat_n(\zeta))<\infty\}\Big]\]
where $\tau(x)=\inf\{k\ge0:X_k=x\}$.
Results   on   Lyapunov exponents  and the quenched level 1 LDP
for nearest-neighbor polymers in  i.i.d.\ random potentials have been proved by 
Carmona and Hu \cite{Car-Hu-04}, Mourrat \cite{Mou-12} and  Zerner \cite{Zer-98-aap}.
Some of the ideas originate in Sznitman  \cite{Szn-94} and 
Varadhan \cite{Var-03-cpam}.

  Our treatment resolves some regularity issues
of the level 1 rate function raised by Carmona and Hu \cite[Remark 1.3]{Car-Hu-04}.  
We require $g$ to be  finite, so for example walks on percolation
clusters are ruled out.   
 Mourrat \cite{Mou-12} proved a  level 1 LDP for simple random walk in an i.i.d.\ potential
$g(\w_0)\le 0$ that permits $g=-\infty$ as long as $g(\w_x)>-\infty$ percolates.  

The directed i.i.d.\ case of Example \ref{ex:dir-iid} in dimension $d=2$,  with a potential $g(\w_0)$ 
subject to some moment assumptions,  
is expected to be a member of the KPZ universality class (Kardar-Parisi-Zhang). 
The universality conjecture is that the centered and normalized  point-to-point free
energy should converge to the Airy$_2$ process.  
 At present such universality remains unattained. 
Piza \cite{Piz-97} proved in some generality that   fluctuations of the point-to-point
free energy diverge at least logarithmically.   
  Among  the lattice models
studied in this paper one is known to be exactly solvable, namely the log-gamma polymer introduced in  \cite{Sep-12-} and further studied in
 \cite{Cor-etal-11-,Geo-Sep-12-}.  For that model the KPZ conjecture
is partially proved:   correct fluctuation exponents were 
verified in some cases in \cite{Sep-12-}, and the Tracy-Widom GUE limit 
  proved in some cases in \cite{Bor-Cor-Rem-12-}.   
KPZ universality results are further along for zero temperature polymers
(oriented percolation or last-passage percolation type models).  
Article   \cite{Cor-11-} is a recent survey of these developments.

 \medskip
 
{\bf Notation and conventions.}    
On a product space  $\Omega=\Gamma^{\Z^d}$
  with generic points $\w=(\w_x)_{x\in\Z^d}$, a {\sl local} function $g(\w)$ is a function
of only finitely many coordinates $\w_x$. 
$\E$ and $\P$ refer to the background measure on the environments $\w$.  
For the  set $\range\subset\Z^d$ of admissible steps we define $M=\max\{\abs z:z\in\range\}$,
and  denote its convex hull in $\R^d$  by $\Uset=\{ \sum_{z\in\range} a_z z: 0\le a_z\in\R,\,\sum_z a_z=1\}$.  
The steps of an admissible path $(x_k)$ are $z_k=x_k-x_{k-1}\in\range$. 

In general, the convex hull of a set $\mathcal I$ is $\conv\mathcal I$.
A convex set   $\mathcal C$ has its relative interior   $\ri\mathcal C$,
its set of extreme points $\ex\mathcal C$,  and its affine hull 
  $\aff\mathcal C$.  
  The upper semicontinuous regularization of a function $f$ is denoted by
$f^{\text{usc}}(x)=\inf_{\text{open}\,B\ni x} \sup_{y\in B} f(y)$ with an analogous definition
for  $f^{\text{lsc}}$.    $E^\mu[f]=\int f\,d\mu$ denotes expectation under the measure $\mu$.   As usual,  $\N=\{1,2, 3,\dotsc\}$ and  $\Z_+=\{0,1,2,\dotsc\}$. 
  $x\vee y=\max(x,y)$ and $x\wedge y=\min(x,y)$.

\section{Existence and regularity of the quenched point-to-point free energy}\label{sec:free energy}

Standing assumptions for this section are that Ê$(\Omega,\kS,\P,\{T_z:z\in\gr\})$Ê
is a measurable dynamical system and $\range$ is finite.
This will not be repeated in the statements of the theorems.
When ergodicity is assumed it is mentioned.  
For the rest of this section we fix the integer $\ell\ge0$.
Define the   space $\bigom_\ell=\Omega\times\range^\ell$.
If $\ell=0$ then  $\bigom_\ell=\Omega$.   Convex analysis will 
be important throughout the paper.   The convex hull of $\range$ is denoted by 
  $\Uset$,   the set of extreme points of $\Uset$ is $\ex\Uset\subset\range$, 
   and  $\ri\Uset$ is the relative interior of $\Uset$.   

  The following is our key
	assumption.  
 
 	\begin{definition}\label{cL-def}
Let $\ell\in\Z_+$.	A function $g:\bigom_\ell\to\R$ is in 
	class $\Ll$ if for each $\tilde z_{1,\ell}\in\range^\ell$ these properties hold: 
	  $g(\cdot\,,\tilde z_{1,\ell})\in L^1(\P)$ and for any nonzero $z\in\range$
	\[\varlimsup_{\e\searrow0}\varlimsup_{n\to\infty} \max_{x\in\gr:\abs{x}\le n}\frac1n \sum_{0\le k\le\e n} 
	\abs{g(T_{x+kz}\w, \tilde z_{1,\ell})}=0\quad\text{for $\P$-a.e.\ $\w$.}\]
	\end{definition}
	
Membership  $g\in\Ll$ depends on a combination of   mixing   of $\P$
and moments of $g$.   If $\P$ is an arbitrary ergodic measure 
then in general we must assume $g$ bounded to guarantee $g\in\Ll$, except 
that if $d=1$ then $g\in L^1(\P)$ is enough.  
  Strong mixing of the process $\{g\circ T_x:x\in\gr\}$   and $g\in L^p(\P)$
  for some  large enough $p$ also guarantee $g\in\Ll$.   
  For example, with exponential mixing 
  $p>d$ is enough.  This is the case in particular if  $g$ has the 
 $r_0$-separated   i.i.d.\ property mentioned in Example \ref{ex:product}.   Lemma A.4 of \cite{Ras-Sep-Yil-12-} gives  a precise statement.

We now define the lattice points  $\xhat_n(\zeta)$ that appear in the point-to-point free
energy  \eqref{pres-8.6}.  
For each  point $\zeta\in\Uset$   fix weights $\beta_z(\zeta)\in[0,1]$ such that 
$\sum_{z\in\range}\beta_z(\zeta) =1$  and 
$\zeta=\sum_{z\in\range}\beta_z(\zeta) z$.  Then define a path 
\begin{align}\label{xhat-def}\xhat_n(\zeta)=\sum_{z\in\range}\bigl(\lfloor n\beta_z(\zeta)\rfloor +b_z^{(n)}(\zeta)\bigr) z, \quad n\in\Z_+,  \end{align}
where  $b_z^{(n)}(\zeta)\in\{0,1\}$ are  arbitrary but subject to these constraints:  
 if $\beta_z(\zeta)=0$ then  $b_z^{(n)}(\zeta)=0$,  and $\sum_z b_z^{(n)}(\zeta) = n-\sum_{z\in\range}\lfloor n\beta_z(\zeta)\rfloor$.   In other words,  $\xhat_n(\zeta)$ is a lattice point 
that approximates $n\zeta$,  is precisely $n$ $\range$-steps away from the origin,
and uses only those steps that appear in the particular  convex representation 
$\zeta=\sum_z \beta_z z$ that was picked.   
When $\zeta\in\Uset\cap\Q^d$ we require that  $\beta_z(\zeta)$ be rational. This is 
  possible by Lemma A.1  of \cite{Ras-Sep-Yil-12-}.
  If we only cared about $\pres_\ell(g,\zeta)$ for rational $\zeta$ we could allow much
  more general paths, see Theorem \ref{th:pt2pt-rats} below. 


The next  theorem establishes the  existence of  the quenched  point-to-point free energy (a) 
and   { free energy} (b).  Introduce the empirical measure $R_n^{\oneell}$ by 
\beq R_n^{\oneell}(g)=n^{-1}\sum_{k=0}^{n-1}g(T_{X_k}\w,Z_{k+1,k+\ell}).\label{Rnell} \eeq

\begin{theorem}\label{th:pt2pt}
Fix  $g\in\Ll$.   

{\rm (a)} For $\P$-a.e.\ $\w$ and simultaneously  for all $\zeta\in\Uset$ the limit
	\begin{align}\label{xipres}
	\pres_\ell(g,\zeta;\w)=
	\lim_{n\to\infty}n^{-1}\log E\big[e^{n R_n^{\oneell}(g)}\one\{X_n=\xhat_n(\zeta)\}\big]
	\end{align}
 exists in $(-\infty,\infty]$.  For a particular $\zeta$ the  limit is  independent of the choice of  convex representation $\zeta=\sum_z\beta_z z$ and  the numbers $b^{(n)}_z$ that define
 $\xhat_n(\zeta)$ in \eqref{xhat-def}.
When $\zeta\not\in\Uset$ it is natural to set $\pres_\ell(g,\zeta)=-\infty$.

{\rm (b)} The limit 
		\begin{align}\label{pres-limit}
		\pres_\ell(g;\w)=\lim_{n\to\infty}n^{-1}\log E\big[e^{\sum_{k=0}^{n-1}g(T_{X_k}\w,Z_{k+1,k+\ell})}\big]
		\end{align}
exists $\P$-a.s.\ in $(-\infty,\infty]$ and satisfies
	\begin{align}\pres_\ell(g)
	=\sup_{\xi\in\Q^d\cap\Uset}\pres_\ell(g,\xi)
	=\sup_{\zeta\in\Uset}\pres_\ell(g,\zeta).\label{sup pt2pt}
	\end{align}
\end{theorem}


Formula \eqref{pres19} in 
  Section \ref{ld-sec}  shows how to   recover
$\pres_\ell(g,\zeta)$ from knowing $\pres_\ell(h)$ for a broad enough class of functions $h$.

\begin{remark}[Conditions for finiteness.]  In general, we need to assume that
$g$ is bounded from above to prevent the possibility that $\pres_\ell(g,\zeta)$ 
takes the value $+\infty$.  When  $g$ has the $r_0$-separated i.i.d.\ property 
  and $0\notin\Uset$  as in Example \ref{ex:dir-iid},  
 the assumption $\E[\abs{g}^p]<\infty$ for some $p>d$ guarantees that $\pres_\ell(g,\zeta)$
and $\pres_\ell(g)$ are a.s.\ finite (Lemma \ref{lm:animal}).  
In fact $\pres_\ell(g,\cdot\,)$ is either bounded or identically $+\infty$ on $\ri\Uset$  
(Theorem \ref{th:regular}).   
\label{finite-rem}\end{remark}

Let us recall  facts about convex sets.
A {\sl face} of a convex set $\Uset$ is a convex subset $\Uset_0$ such that every (closed) line segment in $\Uset$ with a relative interior point in $\Uset_0$ has both 
endpoints in $\Uset_0$. $\Uset$ itself is a face. By Corollary 18.1.3 of \cite{Roc-70} any other face of $\Uset$ is entirely contained in the relative boundary of $\Uset$.
Extreme points of $\Uset$ are the zero-dimensional faces. 
By Theorem 18.2 of \cite{Roc-70}   each point $\zeta\in\Uset$ has a unique face  $\Uset_0$ such that $\zeta\in\ri\Uset_0$.  (An extreme case of this is   $\zeta\in\ex\Uset$ in which
case $\{\zeta\}=\Uset_0=\ri\Uset_0$. Note that the relative interior of a nonempty convex set
is never empty.) 
By Theorem 18.1 of \cite{Roc-70} if $\zeta\in\Uset$ belongs to a face $\Uset_0$ then any representation of $\zeta$ as a convex combination
of elements of $\Uset$  involves only elements of $\Uset_0$. Lastly, Theorem 18.3 in \cite{Roc-70} says that a face $\Uset_0$ is the convex hull of 
$\range_0=\range\cap\Uset_0$.

We address basic properties of $\pres_\ell(g,\zeta;\w)$.  The first issue
is whether it is random (genuinely a function of $\w$) or deterministic
(there is a value $\pres_\ell(g,\zeta)$ such that $\pres_\ell(g,\zeta;\w)=\pres_\ell(g,\zeta)$
for $\P$-almost every $\w$).  
This will depend on the setting.   
   If $0\in\ex\Uset$ then the condition  $X_n=0$ does not permit the walk to
  move and 
  $\pres_\ell(g,0;\w)=-\log|\range|+g(\w,(0,\dotsc,0))$.      
But even if the origin does not cause problems,  
   $\pres_\ell(g,\zeta;\w)$ is not necessarily deterministic on all of $\Uset$ if 
$\P$ is  not   totally ergodic.   
For example, if $0\ne z\in\ex\Uset$ then 
$X_n=nz$ is possible only by repetition of step $z$ and 
$\pres_\ell(g,z;\w)=-\log|\range| +\E[g(\w, (z,\dotsc,z))\,|\,\mathfrak I_z]$, where  $\mathfrak I_z$ is the
$\sigma$-algebra invariant under $T_z$. 

\begin{theorem}\label{th:deterministic}    Fix  $g\in\Ll$. 
Let $\Uset_0$ be any face of $\Uset$, possibly $\Uset$ itself. 
Suppose  $\P$ is ergodic under $\{T_z:z\in\range\cap\Uset_0\}$.
Then there exist a nonrandom function $\pres_\ell(g,\zeta)$  of $\zeta\in\ri\Uset_0$
and an event $\Omega_0$ such that {\rm (i)}  $\P(\Omega_0)=1$ and 
 {\rm (ii)} 
for all $\w\in\Omega_0$ and $\zeta\in\ri\Uset_0$ the limit in \eqref{xipres} equals
$\pres_\ell(g,\zeta)$. 
%
\end{theorem}

\begin{remark} 
(i) For an ergodic $\P$ we get a deterministic function
$\pres_\ell(g,\zeta)$  of $\zeta\in\ri\Uset$.  We write 
$\pres_\ell(g,\zeta;\w)=\pres_\ell(g,\zeta)$ in this case. 

(ii) 
If $\P$ is nondegenerate  
  the  assumption rules out the case $\Uset_0=\{0\}$
because $T_0$ is  the identity mapping.  
$\{0\}$ is   a   face if $0\in\ex\Uset$. 

(iii) 
An important special case is   the  totally ergodic $\P$. Then the
theorem above applies to each face except $\{0\}$.  Since there are only finitely many
faces,  we get a single deterministic function
 $\pres_\ell(g,\zeta)$ and a single event $\Omega_0$ of full $\P$-probability
 such that  $\pres_\ell(g,\zeta)$ is  the limit in \eqref{xipres}  
for all $\w\in\Omega_0$ and $\zeta\in\Uset\smallsetminus\{0\}$. 
The point $\zeta=0$ is included
in this statement  if $0$ is a non-extreme point of $\Uset$. 
\end{remark} 
 
   Convexity of $\pres_\ell(g,\zeta)$ in $g$ follows from H\"older's inequality.
The next theorem establishes some  
  regularity in $\zeta$ for the a.e.\ defined   function
    $\pres_\ell(g,\zeta;\w)$. 
The infinite case needs to be separated.


\begin{theorem} \label{th:regular}
Let  $g\in\Ll$ and assume $\P$ is  ergodic.   
Then $\pres_\ell(g)$ is deterministic.   The following properties
hold for $\P$-a.e.~$\w$. 

{\rm (a)} If $\pres_\ell(g)=\infty$  then $\pres_\ell(g,\zeta)$ is 
 identically $+\infty$ for  $\zeta\in\ri\Uset$.  

{\rm (b)}    Suppose $\pres_\ell(g)<\infty$. Then    $\pres_\ell(g,\cdot\,;\w)$ is  
  lower semicontinuous and bounded on $\Uset$ and  concave and  continuous 
on $\ri\Uset$.  The   
 upper semicontinuous regularization of $\pres_\ell(g,\cdot\,;\w)$ 
 and  its unique continuous extension from $\ri\Uset$ to   $\Uset$ are equal
 and deterministic.  
  \end{theorem}

\begin{remark} 
  Suppose $\P$ is  totally ergodic and we are in the finite case
of Theorem \ref{th:regular}(b). Then concavity in $\zeta$ extends to all of $\Uset$
(see Remark \ref{toterg-rem} below for the argument).  
This is true despite the possibility of a random value  $\pres_\ell(g,0;\w)$ 
at $\zeta=0$ (this happens in the case $0\in\ex\Uset$).  In other words, 
 concavity and  lower semicontinuity are both valid even with the
 random value at $\zeta=0$.  However, continuity must fail because
 on  $\Uset\smallsetminus\{0\}$ the function $\pres_\ell(g, \zeta)$
 is deterministic.  This issue of extending continuity from $\ri\Uset$
 to the boundary is tricky.  We address this issue  in  
  the i.i.d.\ case in Theorem \ref{th:pt2pt cont}.  
 \end{remark}

\medskip

We   turn to the proofs of the theorems in this section. Recall $M=\max\{\abs z:z\in\range\}$.   Let
\beq   D_n=\{z_1+\cdots+z_n:z_{1,n}\in\range^n\}  \label{D_n}\eeq
denote the set of endpoints of admissible paths of length $n$.  
 To prove Theorem \ref{th:pt2pt} we first treat rational points  $\xi\in\Uset$. In this case
we can be more liberal with the function $g$ and with the paths.  

\begin{theorem}\label{th:pt2pt-rats}
Let  $g(\cdot\,,z_{1,\ell})\in L^1(\P)$ for each $z_{1,\ell}\in\range^\ell$.   
Then for $\P$-a.e.\ $\w$ and simultaneously  for all $\xi\in\Uset\cap\Q^d$
the following holds:  for any path   $\{y_n(\xi)\}_{n\in\Z_+}$ such that 
$y_n(\xi)-y_{n-1}(\xi)\in\range$ and for some $k\in\N$,  $y_{mk}(\xi)=mk\xi$ for all $m\in\Z_+$, 
the limit 
	\begin{align}\label{xipresQ}
	\pres_\ell(g,\xi;\w)=
	\lim_{n\to\infty}n^{-1}\log E\big[e^{n R_n^{\oneell}(g)}\one\{X_n=y_n(\xi)\}\big]
	\end{align}
exists in $(-\infty,\infty]$.  For a given $\xi\in\Uset\cap\Q^d$ the limit is independent
of the choice of the path $\{y_n(\xi)\}$ subject to the condition above.  
%
\end{theorem}
 

\begin{theopargself}
\begin{proof}[of Theorem {\ref{th:pt2pt-rats}}.]
Fix $\xi\in\Q^d\cap\Uset$, the path $y_n(\xi)$, and $k$  so that 
$y_{mk}(\xi)=mk\xi$ for all $m\in\Z_+$.
By the Markov property 
\beq\begin{aligned}
 &\log E\big[e^{(m+n)k R^{\oneell}_{(m+n)k}(g)}, X_{(m+n)k}=(m+n)k\xi\big]-2A_\ell(\w)\\
&\qquad\qquad\ge \log E\big[e^{mk R_{mk}^{\oneell}(g)}, X_{mk}=mk\xi\big]-2A_\ell(\w)\\
&\qquad\qquad\qquad+\log E\big[e^{nk R_{nk}^{\oneell}(g\circ T_{mk\xi})}, X_{nk}=nk\xi\big]-2A_\ell(T_{mk\xi}\w),
 \end{aligned}\label{subadd}\eeq
where   $T_x$ acts  by  $g\circ T_x (\w,z_{1,\ell})=g(T_x\w,z_{1,\ell})$
and  the errors are covered by defining 
\[A_\ell(\w)=\ell\max_{y\in\gr: \abs{y}\le M\ell} 
\max_{z_{1,\ell}\in\range^\ell}\max_{1\le i\le\ell} \abs{g(T_{-\xtil_i}\w,z_{1,\ell})} \in L^1(\P).\] 


Since $g\in L^1(\P)$ 
the random variable  $-\log E[e^{nk R_{nk}^{\oneell}(g)}, X_{nk}=nk\xi]+2A_\ell(\w)$ is $\P$-integrable for each $n$.
By Kingman's subadditive ergodic theorem (for example in the form in \cite[Theorem~2.6, page~277]{Lig-05}) 
\begin{align} \pres_\ell(g,\xi;\w)
&=\lim_{m\to\infty}\frac1{mk}\log
 E\big[e^{mk R_{mk}^{\oneell}(g)}, X_{mk}=mk\xi\big] \label{temp5} 
 \end{align}
 exists in $(-\infty,\infty]$ $\P$-almost surely. 
This limit is independent of $k$  because if $k_1$ and $k_2$ both work 
  and give  distinct limits, then  the limit 
along  the subsequence of  multiples
of $k_1k_2$ would not be defined.  Let $\Omega_0$ be  the full probability event 
on which limit \eqref{temp5} holds for all $\xi\in\Q^d\cap\Uset$
and   $k\in \N$  such that 
$k\xi\in\Z^d$.  
 
Next we extend  
  limit  \eqref{temp5} to  the full sequence.  
Given $n$ choose $m$ so that 
$   mk\le n< (m+1)k $.  
By assumption we have admissible paths from 
$mk\xi$ to $y_n(\xi)$ and from $y_n(\xi)$ to $(m+1)k\xi$, 
 so we can create inequalities by restricting the 
expectations to follow these path segments.  
  For convenience let us take $k>\ell$ so that $R^{\oneell}_{(m-1)k}(g)$ 
does not depend on the walk beyond time $mk$.  
Then, for all $\w$
\begin{align}
&\log E\big[e^{n R_{n}^{\oneell}(g)}, X_{n}=y_n(\xi)\big]\nn \\
&\quad\ge 
\log E\big[e^{(m-1)k R^{\oneell}_{(m-1)k}(g)}, X_{mk}=mk\xi,\, X_{n}=y_n(\xi)\big]-A_{2k}(T_{mk\xi}\w)\nn\\
&\quad\ge  \log E\big[e^{(m-1)k R^{\oneell}_{(m-1)k}(g)}, X_{mk}=mk\xi\big]-(n-mk)  \log |\range| -A_{2k}(T_{mk\xi}\w)\nn\\
&\quad\ge  \log E\big[e^{mk R^{\oneell}_{mk}(g)}, X_{mk}=mk\xi\big]
- k\log |\range| -2A_{2k}(T_{mk\xi}\w)\label{temp n+1}  
\end{align}
and similarly 
\begin{align*} 
&\log E\big[e^{(m+1)k R^{\oneell}_{(m+1)k}(g)}, X_{(m+1)k}=(m+1)k\xi\big]\\
&\quad\ge \log E\big[e^{n R^{\oneell}_{n}(g)}, X_{n}=y_n(\xi)\big]
-k\log |\range| -2A_{2k}(T_{mk\xi}\w).  
\end{align*}
Divide by $n$ and take $n\to\infty$ in the bounds developed above.  
Since in general $m^{-1}Y_m\to 0$ a.s.\ for
identically distributed integrable $\{Y_m\}$,  the  error terms vanish in the limit.
The limit holds on the full probability  subset of $\Omega_0$ where the errors 
$n^{-1}A_{2k}(T_{mk\xi}\w)\to 0$ for all $\xi$ and $k$.   We also conclude 
  that  the limit is independent of the choice of the path $y_n(\xi)$.
Theorem \ref{th:pt2pt-rats} is proved.
\qed
\end{proof}
\end{theopargself}

The next lemma will help in the proof of Theorem \ref{th:pt2pt} and the LDP in
Theorem \ref{th:ldp}  

\begin{lemma}\label{lm-pressure2}  
Let  $g\in\Ll$.  Define the paths $\{y_n(\xi)\}$ for $\xi\in \Q^d\cap\Uset$ as in
Theorem \ref{th:pt2pt-rats}.    Then  for $\P$-a.e.\ $\w$, we have the following 
bound  for all 
 compact $K\subset\R^d$ and $\delta>0$:  
 	\begin{align}\label{pres-a}
	&\varlimsup_{n\to\infty}n^{-1}\log E\big[e^{n R_n^{\oneell}(g)}\one\{X_n/n\in K\}\big]\\
	&\qquad\qquad\le\sup_{\xi\in \Q^d\cap K_\delta\cap\Uset}\,\varlimsup_{n\to\infty}n^{-1}\log E\big[e^{n R_n^{\oneell}(g)}\one\{X_n=y_n(\xi)\}\big]
	\end{align}
where $K_\delta=\{\zeta\in\R^d:\exists\zeta'\in K\text{ with }|\zeta-\zeta'|<\delta\}$.
\end{lemma}


\begin{proof}
 Fix a nonzero $\zhat\in\range$.
Fix $\e\in(0,\delta/(4M))$ and  an integer $k\ge|\range|(1+2\e)/\e$. 
There are finitely 
many  points in $k^{-1}D_k$ so we can fix a single 
  integer $b$ such that   $y_{mb}(\xi)=mb\xi$ for all $m\in\Z_+$
and   $\xi\in k^{-1}D_k$. 

We construct a path from   each $x\in D_n\cap nK$ 
  to a  multiple of a  point $\xi(n,x)\in K_\delta\cap k^{-1}D_k$.  
Begin by writing  $x=\sum_{z\in\range} a_z z$ with $a_z\in\Z_+$ and $\sum_{z\in\range}a_z=n$. 
Let $m_n=\ce{(1+2\e)n/k}$ and $s_z^{(n)}=\ce{k a_z/((1+2\e)n)}$. 
Then 
	\[(1-\tfrac1{1+2\e})n^{-1}a_z-\tfrac1k\le n^{-1}a_z - k^{-1}s_z^{(n)}\le(1-\tfrac1{1+2\e})n^{-1}a_z.\]
This implies that \[\tfrac\e{1+2\e}\le 1-k^{-1}\sum_z s^{(n)}_z\le1-\tfrac1{1+2\e}<\tfrac\delta{2M}\]
 and \[\Big| k^{-1}\sum_{z\in\range} s^{(n)}_z z -n^{-1}x\Big|\le M \sum_{z\in\range} |k^{-1}s^{(n)}_z-n^{-1}a_z|\le M(1-\tfrac1{1+2\e})<\tfrac\delta2.\]
Define a point   $\xi(n,x)\in K_\delta\cap k^{-1}D_k$ by  
  \begin{align}\label{xi-steps}
  \xi(n,x)=k^{-1}\sum_{z\in\range} s^{(n)}_z z+\Big(1-k^{-1}\sum_{z\in\range} s^{(n)}_z\Big)\zhat.
  \end{align} 
Since $m_n s^{(n)}_z\ge a_z$ for each $z\in\range$, 
the sum above describes  
an admissible path  of $m_n k-n$ steps   
from $x$ to $m_n k\xi(n,x)$.    For each  $x\in D_n$ and each  $n$,  the number
of  $\zhat$ steps in 
this path is   at least 
\beq  m_n(k-\sum_{z\in\range} s^{(n)}_z)\ge m_n k\e/(1+2\e)\ge n\e.\label{zhat-nr}\eeq

Next, let $\ell_n$ be an integer such that $(\ell_n-1)b<m_n\le \ell_n b$. 
Repeat the steps of $k\xi(n,x)$ in \eqref{xi-steps}  $\ell_n b-m_n\le b$ times to go from $m_n k\xi(n,x)$ to $\ell_nkb\xi(n,x)=y_{\ell_n kb}(\xi(n,x))$.  Thus, the total number of 
steps to go from $x$ to $\ell_n k b\xi(n,x)$ is $r_n=\ell_n k b-n$.   
Recall that $b$ is a function of $k$ alone. So $r_n \le 3\e n$ 
for $n$ large enough, depending on $k, \e$. Denote this sequence of  steps by $\bfu(n,x)=(u_1,\dotsc,u_{r_n})$.

We develop an estimate.   
Abbreviate $\gbar(\w)=$ $\max_{z_{1,\ell}\in\range^\ell} |g(\w,z_{1,\ell})|$.
\beq\begin{aligned}
&\frac1n\log E\big[e^{n R^{\oneell}_n(g)}\one\{X_n/n\in K\}\big]\\ 
&= \frac1n\log \sum_{x\in D_n\cap nK} E\big[e^{n R^{\oneell}_n(g)}, X_n=x\big]\\
&\le \max_{x\in D_n\cap nK}
\frac1n\log   E\big[e^{(n-\ell) R^{\oneell}_{n-\ell}(g)}, X_n=x\big]\\ 
&\qquad\qquad + \max_{x\in D_n\cap nK} \max_{y\in \cup_{s=0}^\ell D_s}\frac {\ell} n  \gbar(T_{x-y}\w)+ \frac{C\log n}{n}\\
&\le \max_{x\in D_n\cap nK}
\frac1n\log   E\big[e^{\ell_n k b R^{\oneell}_{\ell_n kb}(g)}, X_{\ell_n kb}=\ell_nkb\xi(n,x)\big]\\
&\qquad\qquad
+\max_{x\in D_n\cap nK} \frac1n \sum_{i=1}^{r_n} \gbar(T_{x+u_1+\cdots+u_i}\w)
+\frac{r_n}n\log|\range|\\
&\qquad\qquad
+\max_{x\in D_n\cap nK}\max_{y\in \cup_{s=0}^\ell D_s} \frac {2\ell} n  \gbar(T_{x-y}\w)+ \frac{C\log n}{n}.
\end{aligned}\label{est-1}\eeq
As $n\to\infty$ the limsup of the term in the third-to-last line of the above display is 
 bounded above, for all $\w$,  by
\[(1+3\e)\sup_{\xi\in \Q^d\cap K_\delta\cap\Uset}\,\varlimsup_{n\to\infty}n^{-1}\log E\big[e^{n R_n^{\oneell}(g)}\one\{X_n=y_n(\xi)\}\big].\]
The proof of \eqref{pres-a} is complete once we show that a.s.  
\beq\label{to-show}\begin{aligned}
 &\varlimsup_{\e\to0}\varlimsup_{n\to\infty} \max_{x\in D_n} \frac1n\sum_{i=1}^{r_n} \gbar(T_{x+u_1+\cdots+u_i}\w)=0\\
\text{and }\quad&\varlimsup_{\e\to0}\varlimsup_{n\to\infty} \max_{x\in D_n}\max_{y\in \cup_{s=0}^\ell D_s} \frac 1 n  \gbar(T_{x-y}\w)=0.
 \end{aligned} \eeq

To this end, observe that the order in which the steps in $\bfu(n,x)$ are arranged
was so far immaterial. From \eqref{zhat-nr}   the ratio of zero steps to $\zhat$ steps 
is at most  $r_n/(n\e)\le 3$. 
Start path $\bfu(n,x)$ by alternating $\zhat$ steps with
blocks of at most 3 
zero steps, until $\zhat$ steps and  zero steps are exhausted.
After that fix an ordering  $\range\setminus\{0, \zhat\}=\{z_1,z_2,\dotsc\}$ and arrange the rest of the path $\bfu(n,x)$ to take first  all its $z_1$ steps, then all its $z_2$ steps,
and so on. This leads to  the bound
	\beq \sum_{i=1}^{r_n} \gbar(T_{x+u_1+\cdots+u_i}\w)\le 4\abs\range \max_{y\in x+\bfu(n,x)}\max_{z\in\range\setminus\{0\}}\sum_{i=0}^{r_n} \gbar(T_{y+iz}\w).\label{showed}\eeq
The factor 4 is for repetitions of the same $\gbar$-value due to zero steps.
By $y\in x+\bfu(n,x)$ we mean that  $y$ is on the path starting from $x$
and taking steps in $\bfu(n,x)$.   A similar bound develops for the second line of 
\eqref{to-show}.  Then the limits in \eqref{to-show} follow from membership in $\Ll$.
The lemma is proved.
\qed
\end{proof}

\begin{theopargself}
\begin{proof}[of Theorem \ref{th:pt2pt}.]  Part (a).  
Having proved Theorem \ref{th:pt2pt-rats}, the next step is to deduce the existence of $\pres_\ell(g,\zeta)$ as the limit \eqref{xipres}
 for irrational velocities $\zeta$, on the event of full $\P$-probability 
 where $\pres_\ell(g,\xi)$  exists for all rational $\xi\in\Uset$.

Let $\zeta\in\Uset$.  It comes with 
a convex representation $\zeta=\sum_{z\in\range_0}\beta_z z$ with $\beta_z>0$ for $z\in\range_0\subset\range$, and its path $\xhat_\centerdot(\zeta)$ is defined as in \eqref{xhat-def}.
Let  $\delta=\delta(\zeta) =\min_{z\in\range_0}\beta_z>0$. 

We approximate $\zeta$ with rational points from $\conv\range_0$.  
Let $\e>0$ and choose $\xi=\sum_{z\in\range_0}\alpha_z z$ with $\alpha_z\in[\delta/2,1]\cap\Q$, 
$\sum_z\alpha_z=1$, and $|\alpha_z-\beta_z|<\e$ for all $z\in\range_0$. 
%
Let $k\in\N$ be such that $k\alpha_z\in\N$ for all $z\in\range_0$. Let $m_n=\fl{k^{-1}(1+4\e/\delta)n}$ and 
$s_z^{(n)}=km_n\alpha_z-\lfloor n\beta_z\rfloor-b_z^{(n)}$.  Then, 
\beq s_z^{(n)}/n\to(1+4\e/\delta)\alpha_z-\beta_z\ge \e>0.   \label{szn}\eeq 
Thus $s_z^{(n)}\ge0$  for large enough $n$. 

Now, starting at $\xhat_n(\zeta)$ and taking each step $z\in\range_0$ exactly $s_z^{(n)}$ times arrives at 
$km_n\xi$. Denote this sequence  of steps by $\{u_i\}_{i=1}^{r_n}$,
with   $r_n=km_n-n\le (4\e/\delta) n$.   We wish to develop an 
estimate similar to those  in \eqref{temp n+1} and \eqref{est-1},  using again 
$\gbar(\w)=$ $\max_{z_{1,\ell}\in\range^\ell} |g(\w,z_{1,\ell})|$.  
Define 
 \begin{align*}
B(\w,n,\e, \kappa)&=\kappa\abs{\range}\cdot\max_{\abs{x}\le \kappa n} \max_{z\in\range\smallsetminus\{0\}}
 \sum_{i=0}^{\kappa \e n} \gbar(T_{x+iz}\w)  \\
&\qquad\qquad\qquad
 +\max_{x\in D_n}\max_{y\in \cup_{s=0}^\ell D_s}  {2\ell}  \gbar(T_{x-y}\w). 
\end{align*} 
Then develop
an  upper bound:  
\beq\begin{aligned}
&\log E\big[e^{km_n R_{k m_n}^{\oneell}(g)}\one\{X_{km_n}=km_n\xi\}\big]\\
&\qquad \ge \log E\big[e^{n R_n^{\oneell}(g)}\one\{X_n=\xhat_n(\zeta)\}\big] 
 -  \sum_{i=0}^{r_n-1} \gbar(T_{\xhat_n(\zeta)+u_1+\dotsm+u_i}\w)\\
&\qquad\qquad - \max_{y\in \cup_{s=0}^\ell D_s}  {2\ell}  \gbar(T_{\xhat_n(\zeta)-y}\w)
- (4\e/\delta) n\log|\range| \\[3pt]
&  \ge \log E\big[e^{n R_n^{\oneell}(g)}\one\{X_n=\xhat_n(\zeta)\}\big] 
 -  B(\w,n,\e, \kappa)  
- (4\e/\delta) n\log|\range|. \end{aligned}\label{zeta-ub}\eeq
To get  the last inequality above first  order the steps
of the  $\{u_i\}$ path as was done above to go from \eqref{to-show} to \eqref{showed}. 
In particular, the number of zero steps needs to be controlled.   If 
$0\in\range_0$, pick a step $  \zhat\in\range_0\smallsetminus\{0\}$, and from \eqref{szn} obtain that,
for large enough $n$, 
\[  \frac{s^{(n)}_0}{s^{(n)}_\zhat}\le \frac{2n\bigl((1+4\e/\delta)\alpha_0-\beta_0\bigr)}{n\e/2}
 \le 4\Bigl(1+\frac4\delta\Bigr). 
  \]
Thus we  can exhaust the zero steps by alternating blocks of $\ce{4(1+4/\delta)}$ 
zero steps with individual $\zhat$ steps.  Consequently in the sum on the second line 
of \eqref{zeta-ub}    we have a bound $c(\delta)$ on the number
of repetitions of individual $\gbar$-values.    To realize the domination by 
$B(\w,n,\e, \kappa)$ on the last line of \eqref{zeta-ub},  
pick $\kappa>c(\delta)$ and large enough so that $\kappa\e n\ge r_n$ and so that  $\{\abs{x}\le \kappa n\}$
covers   $\{\xhat_n(\zeta)+u_1+\dotsm+u_i: 0\le i\le r_n\}$. 

The point of formulating the error $B(\w,n,\e, \kappa)$ with the parameter $\kappa$
 is to control all  the  errors in \eqref{zeta-ub}  on a single event of $\P$-measure 1,
 simultaneously for all $\zeta\in\Uset$ and 
countably many $\e\searrow0$, with a choice of rational $\xi$ for each pair $(\zeta, \e)$. 
From $g\in\Ll$    follows that  $\P$-a.s. 
\[ \varlimsup_{\e\searrow0}\varlimsup_{n\to\infty}  n^{-1}B(\w,n,\e, \kappa)= 0 
\quad\text{ simultaneously for all $\kappa\in\N$.} \]


A similar argument, with $\bar m_n=\lfloor k^{-1}(1-4\e/\delta)n\rfloor$ and $\bar s_z^{(n)}=\lfloor n\beta_z\rfloor+b_z^{(n)}(\zeta)-k\bar m_n\alpha_z$,   gives
\beq\begin{aligned}
&\log E\big[e^{k\bar m_n R_{k \bar m_n}^{\oneell}(g)}\one\{X_{k\bar m_n}=k\bar m_n\xi\}\big]\\
& \le \log E\big[e^{n R_n^{\oneell}(g)}\one\{X_n=\xhat_n(\zeta)\}\big] + C\e n\log|\range| +B(\w,n,\e,\kappa). 
\end{aligned}\label{zeta-lb}\eeq

Now in \eqref{zeta-ub} and \eqref{zeta-lb} divide by $n$, let $n\to\infty$ and use the 
existence of the limit $\pres_\ell(g,\xi)$.  Since $\e>0$ can be taken
to zero,  we have obtained the following.     $\pres_\ell(g,\zeta)$ exists as the limit
\eqref{xipres}  for all 
$\zeta\in\Uset$ on an event of $\P$-probability $1$, and   
  \beq	\pres_\ell(g,\zeta)=\lim_{\xi_j\to\zeta}\pres_\ell(g,\xi_j),\label{pres-5}\eeq
whenever $\xi_j$ is a sequence of rational  convex combinations of $\range_0$ whose coefficients   converge to the coefficients $\beta_z$  of $\zeta$.

At this point the value $\pres_\ell(g,\zeta)$ appears to depend on the choice of the
convex representation  $\zeta=\sum_{z\in\range_0}\beta_z z$.  
We show that each choice   gives the same
value $\pres_\ell(g,\zeta)$ as a particular fixed representation.  
Let $\bar\Uset$ be the unique face containing
$\zeta$ in its relative interior and  $\bar\range=\range\cap\bar\Uset$.
Then we can fix a convex representation
$\zeta=\sum_{z\in\bar\range}\bar\beta_z z$ with $\bar\beta_z>0$ for all $z\in\bar\range$. 
As above, let $\xi_n$ be rational points from $\conv\range_0$ such that 
  $\xi_n\to\zeta$.  The fact that $\zeta$ can be expressed as a convex combination
  of $\range_0$  forces $\range_0\subset\bar\Uset$, and consequently 
$\xi_n\in\bar\Uset$. By Lemma \ref{lm:co1}, there are two   rational 
convex representations
$\xi_n=\sum_{z\in\range_0}\alpha^n_z z=\sum_{z\in\bar\range}\bar\alpha^n_z z$ with $\alpha_z^n\to\beta_z$ and $\bar\alpha_z^n\to\bar\beta_z$.
By Theorem \ref{th:pt2pt-rats} the value  $\pres_\ell(g,\xi_n)$ is independent of the convex representation of $\xi_n$.  Hence the limit in \eqref{pres-5} shows 
that  representations in terms of $\range_0$ and  in terms of $\bar\range$ 
lead to the same value $\pres_\ell(g,\zeta)$.

Part (b).   With the limit \eqref{xipres} in hand, 
  limit \eqref{pres-limit} and the variational formula  
  \eqref{sup pt2pt} follow from Lemma \ref{lm-pressure2} with $K=\Uset$. Theorem \ref{th:pt2pt} is proved.
\qed  
\end{proof}
\end{theopargself}




Proofs of the remaining theorems of the section follow. 

\begin{theopargself}
\begin{proof}[of Theorem \ref{th:deterministic}]
Fix a face $\Uset_0$ and  $\range_0=\range\cap\Uset_0$.
If $\xi$ is a rational point in $\ri\Uset_0$,  then write $\xi=\sum_{z\in\range_0}\alpha_z z$ with rational $\alpha_z>0$ (consequence of Lemma A.1 of \cite{Ras-Sep-Yil-12-}). 
Let $k\in\N$   such that $k\alpha_z\in\Z$ for
each $z$. Let  $z\in\range_0$. There is a path of $k-1$ steps
from $(m-1)k\xi+z$ to $mk\xi$. 
Proceed as   in \eqref{temp n+1} to reach 
\begin{align*}
&\pres_\ell(g,\xi)
\ge\varliminf_{m\to\infty}\frac1{mk}\log
 E\Big[e^{mk R_{mk}^{\oneell}(g)}, X_{mk}=mk\xi\,\Big|\,X_1=z\Big]\\
&\ge\varliminf_{m\to\infty}\frac1{mk}\log
 E\Big[e^{((m-1)k+1) R_{(m-1)k+1}^{\oneell}(g)},\\
 &\qquad\qquad\qquad\qquad\qquad\qquad\qquad X_{(m-1)k+1}=(m-1)k\xi+z\,\Big|\,X_1=z\Big]\\
&=\pres_\ell(g,\xi)\circ T_z.
\end{align*}
Thus  $\pres_\ell(g,\xi)$ is $T_z$-invariant for each  $z\in\range_0$ so by ergodicity 
$\pres_\ell(g,\xi)$ is deterministic. This holds for $\P$-a.e.\ $\w$ simultaneously for all rational $\xi\in\ri\Uset_0$.
Since $\pres_\ell(g,\cdot)$ at irrational points of $\ri\Uset_0$ can be obtained through \eqref{pres-5} from its values at rational points, the claim follows for all
 $\zeta\in\ri\Uset_0$. 
\qed
\end{proof}
\end{theopargself}


\begin{theopargself}
\begin{proof}[of Theorem \ref{th:regular}]
The logical order of the proof is not the same as the ordering of the statements
in the theorem.  
 First we establish concavity for rational points in $\ri\Uset$ via  the Markov property.
For  $t\in\Q\cap[0,1]$ and $\xi',\xi''\in\Q^d\cap\ri\Uset$ 
  choose $k$ so that $kt\in\Z_+$, $kt\xi'\in\Z^d$, and $k(1-t)\xi''\in\Z^d$. Then, as   in \eqref{subadd},
\beq\begin{aligned}
&\log E\Big[e^{mk R_{mk}^{\oneell}(g)}, X_{mk}=mk(t\xi'+(1-t)\xi'')\Big]\\
&\quad\ge \log E\Big[e^{mkt R_{mkt}^{\oneell}(g)}, X_{mkt}=mkt\xi'\Big]\\
&\qquad\quad+ \log E\Big[e^{mk(1-t) R_{mk(1-t)}^{\oneell}(g\circ T_{mkt\xi'})}, X_{mk(1-t)}=mk(1-t)\xi''\Big]\\
&\qquad\quad-2A_\ell(T_{mkt\xi'}\w).
\end{aligned}\label{conc2}\eeq
Divide by $mk$ and let $m\to\infty$.  
 On $\ri\Uset$   $\pres_\ell(g,\cdot)$ is deterministic  
(Theorem \ref{th:deterministic}),  hence  the second (shifted)  logarithmic moment generating 
function on the right of \eqref{conc2} converges to its limit at least in probability, hence a.s. along
a subsequence.   
In the limit we  get  
\begin{align}\label{concavity}
\pres_\ell(g,t\xi'+(1-t)\xi'')\ge t\pres_\ell(g,\xi')+(1-t)\pres_\ell(g,\xi'').
\end{align} 

To get concavity on all of $\ri\Uset$,  approximate 
arbitrary points of $\ri\Uset$ with rational convex combinations
  so that   limit   \eqref{pres-5}
can be used to pass along the concavity.  



\begin{remark}
 In the totally ergodic case    Theorem \ref{th:deterministic}   implies that  $\pres_\ell(g,\zeta)$ is deterministic on all of $\Uset$, except possibly at  $\zeta=0\in\ex\Uset$.  
 If $0$ is among $\{\xi',\xi''\}$  then  take $\xi'=0$  in   \eqref{conc2}, so that, 
as the limit is taken
to go from \eqref{conc2} to \eqref{concavity},   we can take advantage of the deterministic limit 
$\pres_\ell(g,\xi'')$ for 
  the shifted term 
on the right of \eqref{conc2}. 
Thus, \eqref{concavity} holds for all rational  $\xi',\xi''\in\Uset$.  The subsequent limit to 
non-rational points proceeds as above.  
\label{toterg-rem}\end{remark}

\medskip

Next  we address lower semicontinuity of $\pres_\ell(g,\zeta)$ in $\zeta\in\Uset$.
Fix   $\zeta$ and  pick  $\Uset\ni\zeta_j\to\zeta$
that achieves the liminf of $\pres_\ell(g,\cdot)$ at $\zeta$. Since $\range$ is finite, one can find a further subsequence that always stays inside the convex hull
$\Uset_0$ of some set $\range_0\subset\range$ of at most $d+1$ affinely independent vectors. Then, $\zeta\in\Uset_0$ and we can write the convex
combinations $\zeta=\sum_{z\in\range_0}\beta_z z$ and $\zeta_j=\sum_{z\in\range_0}\beta_z^{(j)} z$.  
Furthermore, as before, $\beta_z^{(j)}\to\beta_z$ as $j\to\infty$. Let $\hat\range_0=\{z\in\range_0:\beta_z>0\}$ and define
$\delta=\min_{z\in\hat\range_0}\beta_z>0$.  

Fix $\e\in(0,\delta/2)$ and take $j$ large enough so that $|\beta_z^{(j)}-\beta_z|<\e$ for all $z\in\range_0$. Let $m_n=\lceil (1+4\e/\delta)n\rceil$ and
$s_z^{(n)}=\lfloor m_n\beta_z^{(j)}\rfloor + b_z^{(n)}(\zeta_j)-\lfloor n\beta_z\rfloor-b_z^{(n)}(\zeta)$ for $z\in\range_0$. 
(If $\beta_z=\beta_z^{(j)}=0$, then simply set $s_z^{(n)}=0$.) Then, for $n$ large enough, $s_z^{(n)}\ge0$ for each $z\in\range_0$.   
Now, proceed as in the proof of \eqref{pres-5}, by finding a path from $\xhat_n(\zeta)$ to $\xhat_{m_n}(\zeta_j)$. After taking $n\to\infty$, $j\to\infty$, then $\e\to0$,
we arrive at \[\varliminf_{\Uset\ni\zeta'\to\zeta}\pres_\ell(g,\zeta')\ge\pres_\ell(g,\zeta).\]
Note that here random limit values are perfectly acceptable.

\begin{remark}
We can see here why upper semicontinuity (and hence continuity to the boundary) may in principle not hold: constructing a path from $\zeta_j$ to $\zeta$ is not
necessarily possible since $\zeta_j$ may have non-zero components on $\range_0\smallsetminus\hat\range_0$.
 \end{remark}

By lower semicontinuity  the supremum   in \eqref{sup pt2pt} 
can be restricted to $\zeta\in\ri\Uset$.   By Theorem \ref{th:deterministic} 
$\pres_\ell(g,\zeta)$ is deterministic on $\ri\Uset$   under an ergodic  $\P$, and consequently 
$\pres_\ell(g)$ is deterministic.  

\medskip

Combining Theorems \ref{th:pt2pt} and \ref{th:deterministic} 
 and the paragraphs above, we now know
that under an ergodic $\P$,  we have the function 
$-\infty< \pres_\ell(g,\zeta,\w)\le \infty$,  $\P$-a.e. 
defined, lower semicontinuous for $\zeta\in\Uset$ and concave and
deterministic  for  
$\zeta\in\ri\Uset$. 
Lower semicontinuity and   compactness of   $\Uset$   imply that  
$\pres_\ell(g,\cdot\,,\w)$ is uniformly bounded below with a bound 
that can depend
on $\w$.  

 Assume now that $\pres_\ell(g)<\infty$.   Then upper boundedness 
 of  $\pres_\ell(g,\cdot\,,\w)$ comes from \eqref{sup pt2pt}.  
 As a finite concave function $\pres_\ell(g,\cdot)$ is continuous on the
convex open set $\ri\Uset$.    Since it is bounded below,
by \cite[Theorem~10.3]{Roc-70}   $\pres_\ell(g,\cdot)$
has a unique continuous extension from the relative interior to the whole of $\Uset$. 
This extension is deterministic since it comes from a deterministic function
on $\ri\Uset$. 
To see that this extension agrees with the    upper semicontinuous regularization, 
consider this general situation. 

Let  $f$ be a bounded lower semicontinuous function
 on $\Uset$ that is concave on $\ri\Uset$. Let $g$ be the continuous extension
 of $f\vert_{\ri\Uset}$  and $h$ the upper semicontinuous regularization of $f$ on $\Uset$.
For $x$ on the relative  boundary  find $\ri\Uset\ni x_n\to x$. Then  
$g(x) = \lim g(x_n) = \lim f(x_n) \ge f(x)$ and    
so $f \le  g$ and consequently $h \le g$.
Also
$g(x)= \lim g(x_n) = \lim f(x_n) = \lim h(x_n) \le h(x)$ and so 
$g \le h$.
 
 \medskip

Finally we check part (a) of the theorem.  
If $\pres_\ell(g)=\infty$ then there exists a sequence $\zeta_n\in\ri\Uset$ 
such that $\pres_\ell(g,\zeta_n)\to\infty$. One can assume $\zeta_n\to\zeta\in\Uset$.
Let $\zeta'$ be any point in  $\ri\Uset$. Pick $t\in(0,1)$ small enough for $\zeta''_n=(\zeta'-t\zeta_n)/(1-t)$ to be in $\ri\Uset$ for $n$ large enough. Then, 
\[\pres_\ell(g,\zeta')\ge t\pres_\ell(g,\zeta_n)+(1-t)\pres_\ell(g,\zeta''_n).\]
Since $\pres_\ell(g,\cdot)$ is bounded below on $\ri\Uset$, taking $n\to\infty$ in the above display implies that $\pres_\ell(g,\zeta')=\infty$.
 \qed
\end{proof}
\end{theopargself}

\section{Continuity   in the  i.i.d.\ case} 
\label{sec:cont} 

We begin with   $L^p$ continuity of the free energy in the
potential $g$.

\begin{lemma} \label{lm:animal}
Let $\Uset_0$ be a face of $\Uset$ (the choice $\Uset_0=\Uset$ is allowed),
and let $\range_0=\range\cap\Uset_0$ so that $\Uset_0=\conv\range_0$.  
Assume  
$0\not\in\Uset_0$.   Then  an  admissible $n$-step 
path from $0$ to a point 
in $n\Uset_0$ cannot  visit the same point twice. 

  {\rm (a)} 
Let $h\ge 0$ be a measurable function on $\Omega$ with  
  the   $r_0$-separated i.i.d.\ property.   
Then there is a constant $C=C(r_0,d, M)$   such that,  $\P$-almost surely,  
\beq
\varlimsup_{n\to\infty}   \max_{\substack{x_{0,n-1}: \\ x_k-x_{k-1}\in\range_0}}  n^{-1}
\sum_{k=0}^{n-1}  h(T_{x_k}\w)  
\le  C   \int_0^\infty \P\{ h\ge s\}^{1/d}\,ds .
\label{animal1}\eeq
If   $h\in L^p(\P)$ for some  $p>d$ then 
the right-hand side of \eqref{animal1} is finite by Chebyshev's inequality.

{\rm (b)}  Let $f, g:\bigom_\ell\to\R$ be  measurable functions with the   $r_0$-separated i.i.d.\ property. Then with the
same constant $C$ as in \eqref{animal1} 
\beq\begin{aligned}
&\varlimsup_{n\to\infty}   \sup_{\zeta\in\Uset_0}\, \Bigl\lvert 
n^{-1}\log E\big[e^{nR_{n}^{\oneell}(f)}\one\{X_{n}=\xhat_n(\zeta)\}\big]\\
&\qquad\qquad\qquad \qquad 
\;-\; n^{-1}\log E\big[e^{nR_{n}^{\oneell}(g)}\one\{X_{n}=\xhat_n(\zeta)\}\big]
\Bigr\rvert \\
&\qquad \le   C   \int_0^\infty \P\Bigl\{\w: 
 \max_{z_{1,\ell}\in\range^\ell}\abs{f(\w,z_{1,\ell})-g(\w,z_{1,\ell})}\ge s\Bigr\}^{1/d}\,ds .
\end{aligned}\label{animal2}\eeq
 Assume  additionally that  $f(\cdot\,,z_{1,\ell})$, $g(\cdot\,,z_{1,\ell})\in L^p(\P)$ $\forall z_{1,\ell}\in\range^\ell$
for some  $p>d$.  Then  $f,g\in\Ll$ 
and  for $\zeta\in\Uset_0$  the limits 
 $\pres_\ell(f,\zeta)$ and $\pres_\ell(g,\zeta)$
are finite and deterministic  and satisfy 
\beq
 \sup_{\zeta\in\Uset_0}\, \abs{\pres_\ell(f,\zeta)-\pres_\ell(g,\zeta)} \, \le \,  
C\E \Bigl[\; \max_{z_{1,\ell}\in\range^\ell}\abs{f(\w,z_{1,\ell})-g(\w,z_{1,\ell})}^p\,\Bigr].  
\label{animal3}\eeq
Strengthen the assumptions further with $0\notin\Uset$.  
Then  $\pres_\ell(f)$ and  $\pres_\ell(g)$ are finite and deterministic  and satisfy 
 \beq
 \abs{\pres_\ell(f)-\pres_\ell(g)} \, \le \,  
C\E \Bigl[\; \max_{z_{1,\ell}\in\range^\ell}\abs{f(\w,z_{1,\ell})-g(\w,z_{1,\ell})}^p\,\Bigr]. 
\label{animal4}\eeq
\label{lm:animal1}\end{lemma}

\begin{proof}   If $x\in n\Uset_0$ and $x=\sum_{i=1}^n z_i$ gives an admissible
path to $x$, then  $n^{-1}x=n^{-1} \sum_{i=1}^n z_i$ gives a convex 
representation of $n^{-1}x\in\Uset_0$ which then cannot use points $z\in\range\smallsetminus\range_0$.  By the assumption $0\notin\Uset_0$, points from $\range_0$ cannot sum to $0$
and consequently a loop in an $\range_0$-path is impossible.

\medskip

 Part (a).  We can assume that $r_0>M=\max\{\abs z:z\in\range\}$.  
We bound the quantity on the left of \eqref{animal1} with a greedy lattice animal
\cite{Cox-Gan-Gri-93,Gan-Kes-94,Mar-02}
after a suitable coarse graining of the lattice.   Let $B=\{0,1,\dotsc, r_0-1\}^d$ be
the cube whose copies $\{r_0y+B: y\in\Z^d\}$ tile the lattice.  
Let $\cA_n$ denote the set of connected subsets $\xi$ of $\Z^d$ of size $n$ that
contain the origin (lattice animals).   

Since the  $x_k$'s are distinct,   
\begin{align*}
&\sum_{k=0}^{n-1}  h(T_{x_k}\w)  =  \sum_{u\in B}\sum_{y\in \Z^d}
\sum_{k=0}^{n-1}  \one_{\{x_k=r_0y+u\}}  h(T_{r_0y+u}\w) \\
&\qquad \le   \sum_{u\in B}\sum_{y\in \Z^d}
  \one_{\{x_{0,n-1}\cap (r_0y+B)\ne \emptyset\}}  h(T_{u+r_0y}\w) \\
&\qquad \le   \sum_{u\in B} \max_{\xi\in\cA_{n(d-1)}} \sum_{y\in \xi}  h(T_{u+r_0y}\w).  
\end{align*}
The   last step works as follows.  Define first a vector $y_{0,n-1}\in(\Z^{d})^n$ from
the conditions $x_i\in r_0y_i+B$, $0\le i<n$.  
Since $r_0$ is larger than the maximal step size $M$,
$\abs{y_{i+1}-y_i}_\infty\le 1$.     Points $y_i$ and $y_{i+1}$ may fail to be nearest neighbors,
but by filling in at most $d-1$ intermediate points we get a nearest-neighbor sequence.
This sequence can have repetitions and  can have fewer than $n(d-1)$ entries,  but it 
 is contained in some lattice animal $\xi$ of $n(d-1)$ lattice points.  

We can assume that the right-hand side of \eqref{animal1} is finite.  This and
the fact that   
$\{ h(T_{u+r_0y}\w)   : y\in \Z^d\}$ are i.i.d. allows us to apply 
  limit (1.7) of Theorem 1.1 in \cite{Mar-02}:  for a finite constant $c$ and 
$\P$-a.s.
\[   \varlimsup_{n\to\infty}   \max_{\substack{x_{0,n-1}: \\ x_k-x_{k-1}\in\range_0}}  n^{-1}
\sum_{k=0}^{n-1}  h(T_{x_k}\w) 
\le  \abs{B} (d-1) c  \int_0^\infty \P\{ h\ge s\}^{1/d}\,ds .
\]
With the volume $\abs{B}=r_0^d$ this gives  \eqref{animal1}.  

\medskip

Part (b). 
 Write $f=g+(f-g)$ in the exponent to get  an estimate, uniformly in $\zeta\in\Uset_0$: 
\beq\begin{aligned}
&n^{-1}\log E\big[e^{nR_{n}^{\oneell}(f)}\one\{X_{n}=\xhat_n(\zeta)\}\big]\\
&\qquad \le 
  n^{-1}\log E\big[e^{nR_{n}^{\oneell}(g)}\one\{X_{n}=\xhat_n(\zeta)\}\big] \\
&\qquad\quad+  \max_{\substack{x_{0,n+\ell-1}: \\ x_k-x_{k-1}\in\range_0}}   n^{-1}
\sum_{k=0}^{n-1}  \abs{f(T_{x_k}\w, z_{k+1,k+\ell})  - g(T_{x_k}\w, z_{k+1,k+\ell}) }. 
\end{aligned}\label{animal8}\eeq 
Switch the roles of $f$ and $g$ to get a bound on the absolute difference.  Apply part (a) to get \eqref{animal2}.

By Lemma A.4 of \cite{Ras-Sep-Yil-12-}  the $L^p$ assumption
with $p>d$    implies that $f, g\in\cL$.    Finiteness of $\pres_\ell(f,\zeta)$ 
comes from \eqref{animal2} with $g=0$. 
 Chebyshev's inequality bounds the right-hand side of  \eqref{animal2} with
 the right-hand side of  \eqref{animal3}. 

To get \eqref{animal4}  start with \eqref{animal8} without the indicators inside the
expectations and with $\range_0$ replaced by $\range$.   
 \qed
\end{proof}

Next the continuity of $\pres_\ell(g,\zeta)$ as a 
function of $\zeta$ all the way to the relative boundary in the i.i.d.\ case. 
The main result is part (a) below. Parts (b) and (c) come without extra work.  

\begin{theorem}\label{th:pt2pt cont}
Let $\P$ be an i.i.d.\ product measure as described in Example \ref{ex:product}
and $p>d$.   
Let $g:\bigom_\ell\to\R$ be a function such that  for each $z_{1,\ell}\in\range^\ell$, 
  $g(\cdot,z_{1,\ell}) $ is a local function of $\w$ and a member of $L^p(\P)$.  
 
{\rm (a)} If $0\not\in\Uset$, then $\pres_\ell(g,\zeta)$ is continuous on $\Uset$.

{\rm (b)} If $0\in\ri\Uset$ and $g$ is bounded above, then $\pres_\ell(g,\zeta)$ is continuous on $\Uset$.

{\rm (c)} If $0$ is on the relative boundary of $\Uset$  and if $g$ is bounded above, then
$\pres_\ell(g,\zeta)$ is continuous on $\ri\Uset$, at nonzero extreme points of $\Uset$, and at any point $\zeta$ such that the face $\Uset_0$  satisfying $\zeta\in\ri\Uset_0$
does not contain $\{0\}$. 
\end{theorem}

In (b) and (c) we assume $g$ bounded above because otherwise   $\pres_\ell(g)=\infty$ is possible. If $g$ is unbounded above and a  function of $\w$ alone and if admissible paths can form loops, then $\pres_\ell(g)=\infty$ because the walk can look for arbitrarily high values of $g(T_x\w)$ and keep returning to  $x$ 
forever.  Then by Theorem \ref{th:regular}(a) also $\pres_\ell(g,\zeta)=\infty$ for all $\zeta\in\ri\Uset$.

In certain situations our proof technique can be pushed   up to faces that
 include $0$. For example, 
for $\range=\{(1,0),(0,1),(0,0)\}$  $\pres_\ell(g,\zeta)$ is continuous in $\zeta\in\Uset\smallsetminus\{0\}$.

\begin{theopargself}
\begin{proof}[of Theorem {\ref{th:pt2pt cont}}]
This continuity argument was inspired by the treatment of the case
$\range=\{e_1,\dotsc, e_d\}$ in \cite{Mor-10-} and \cite{Geo-11}.

By Lemma A.4 of \cite{Ras-Sep-Yil-12-}  the $L^p$ assumption
with $p>d$    implies that $g\in\cL$. 
By Lemma \ref{lm:animal} in case (a),  and by the upper bound assumption in the
other cases,  
  $\pres_\ell(g)<\infty$.  Thereby    $\pres_\ell(g,\cdot)$ is bounded  on $\Uset$ 
  and continuous on $\ri\Uset$  
(Theorem \ref{th:regular}).  
Since $\pres_\ell(g,\cdot)$ is lower semicontinuous, it suffices to prove upper semicontinuity at the relative boundary of $\Uset$.
Let $\zeta$ be a point on the relative boundary of $\Uset$.

We begin by reducing the proof to the case of a bounded $g$.  We can approximate $g$
in $L^p$ with a bounded function.  In part (a) we can apply \eqref{animal3} to 
$\Uset_0=\Uset$.  Then  the uniformity in $\zeta$  of \eqref{animal3} implies that
it suffices to prove  upper semicontinuity in the case of bounded $g$.   In parts (b) and (c) 
$g$ is bounded above to begin with. 
Assume that upper semicontinuity has been
proved for the bounded truncation $g_c=g\vee c$.  Then
\[\varlimsup_{\zeta'\to\zeta}\pres_\ell(g,\zeta')\le
\varlimsup_{\zeta'\to\zeta}\pres_\ell(g_c,\zeta')
\le\pres_\ell(g_c,\zeta).\]
In cases (b) and (c) the unique face $\Uset_0$
that contains $\zeta$ in its relative interior does not contain $0$,
and we can apply \eqref{animal3} to show that $\pres_\ell(g_c,\zeta)$ decreases to $\pres_\ell(g,\zeta)$ which proves upper semicontinuity for $g$.
We can now assume $g$ is bounded, and by subtracting a constant we can assume $g\le0$.

We only prove upper semicontinuity away from the extreme points of $\Uset$. The argument for the extreme points of $\Uset$ is an easier version of the proof.
Assume thus that  the point $\zeta$ on the boundary of $\Uset$  is not an extreme point. Let $\Uset_0$ be the unique face of $\Uset$ such that $\zeta\in\ri\Uset_0$.
Let $\range_0=\range\cap\Uset_0$.  Then  $\Uset_0=\conv\range_0$ and any 
convex representation $\zeta=\sum_{z\in\range}\beta_zz$ of $\zeta$ can only
use $z\in\range_0$ \cite[Theorems~18.1 and 18.3]{Roc-70}.  


The theorem   follows if we show that for any fixed $\delta>0$ and $\xi\in\Q^d\cap\Uset$ close enough to $\zeta$ and for $k\in\N$ such that $k\xi\in\Z^d$,  
\begin{align}\label{continuity-goal}
\lim_{m\to\infty}\P\Big\{\sum_{x_{0,mk+\ell}\in\Pi_{mk,mk\xi}}\!\!\!\!\!\!\! e^{mkR_{mk}^{\oneell}(g)}\ge e^{mk(\pres_\ell(g,\zeta)+\log|\range|)+6mk\delta}\Big\}=0.
\end{align}
Here we used the  approximation by rational points \eqref{pres-5}.
 $\Pi_{mk,mk\xi}$ is the set of admissible paths $x_{0,mk+\ell}$
such that $x_0=0$ and $x_{mk}=mk\xi$.  
It is enough to approach $\zeta$ 
from outside  $\Uset_0$ because continuity on $\ri\Uset_0$ is guaranteed by 
concavity.
Fix $\delta>0$.
 
Since $0\notin\Uset_0$ we can find a vector  $\uhat\in\Z^d$ such that 
$z\cdot\uhat>0$ for  $z\in\range_0$. 
 
\begin{figure}[h]
\begin{center}
\includegraphics[width=0.7\textwidth]{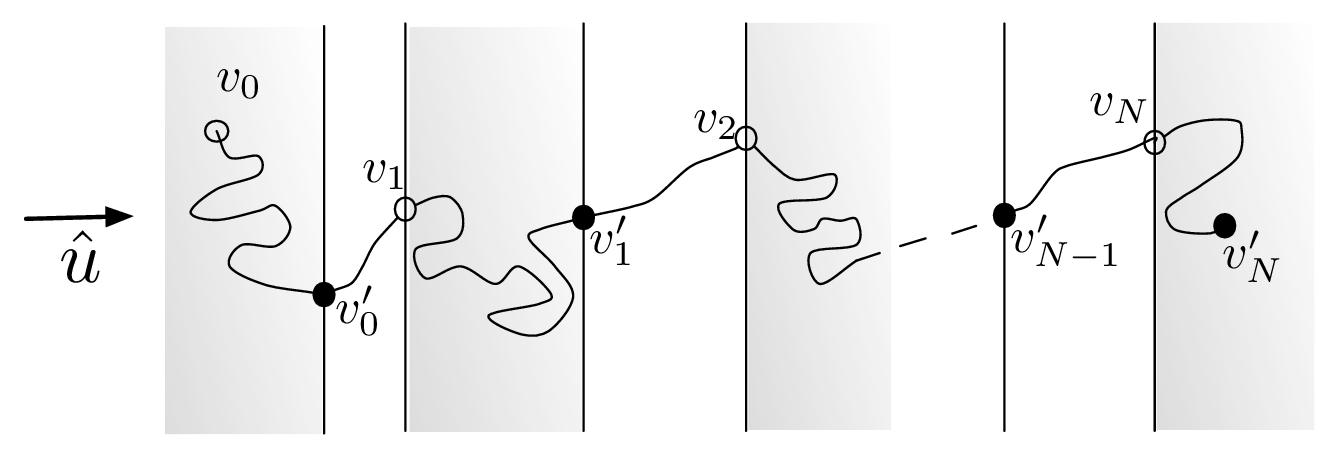}
\end{center}
\caption{Path segments in shaded regions are bad, the other segments are good. $v_i=X_{s_i}$ and $v_i'=X_{s_i'}$. Steps going up and to the right 
represent steps in $\range_0$.}\label{fig1}
\end{figure}
 
Given a path $x_{0,mk+\ell}$  let $s_0=0$ and, if it exists, let $s'_0\ge 0$ be its first {\sl regeneration time}: this is the first time $i\in[0,mk]$ such that 
$x_j\cdot\uhat\le x_i\cdot\uhat$ for $j\le i$, $z_{i+1,i+\ell}\in\range_0^\ell$, and $x_j\cdot\uhat>x_{i+\ell}\cdot\uhat$ for $j\in\{i+\ell+1,\dotsc, mk+\ell\}$.
If $s'_0$ does not exist then we set $s'_0=mk+\ell$ and stop at that. Otherwise, if $s'_0$ exists, then let 
\begin{align*}
s_1=\min\{j\in&(s'_0,mk+\ell):z_{j+1}\not\in\range_0\\
&\text{ or }\exists i\in(j+1,mk+\ell]\text{ such that }x_i\cdot\uhat\le x_{j+1}\cdot\uhat\}.
\end{align*}
If such a time does not exist, then we set $s_1=s'_1=mk+\ell$ and stop. Otherwise,
define $s_1<s'_1<s_2<s'_2<\dotsm$ inductively. Path segments  $x_{s'_i, s_{i+1}}$ are
 {\sl good}  and segments 
 $x_{s_i,s_i'}$ are {\sl bad} (the paths in the gray blocks in Figure \ref{fig1}). 
Good segments  have length at least $\ell$ and consist of only $\range_0$-steps, 
 and distinct  good segments lie in disjoint  slabs (a {\sl slab} is a portion of $\Z^d$ 
 between two hyperplanes perpendicular to $\uhat$).  
  
Time $mk+\ell$ may belong to an incomplete bad segment and then in the above procedure
the last time defined was $s_N<mk+\ell$ for some $N\ge0$ and we set $s'_{N}=mk+\ell$, 
or to a good segment in which case 
the last time defined was $s'_{N-1}\le mk$ for some $N\ge1$ and we set $s_{N}=s'_{N}=mk+\ell$. 
There are $N$    good segments   and $N+1$  bad segments, when we 
admit  possibly  degenerate first and last bad segments $x_{s_0,s'_0}$
 and $x_{s_N,s'_N}$  
(a degenerate segment has no steps). 
  Except possibly for   $x_{s_0,s'_0}$
 and $x_{s_N,s'_N}$, each  bad segment has at least 
 one $(\range\smallsetminus\range_0)$-step.


\begin{lemma} Given $\e>0$, we can choose $\e_0\in(0,\e)$ such that  if 
  $|\xi-\zeta|<\e_0$, then the total number of steps in the bad segments in any path in $\Pi_{mk,mk\xi}$ is at most
$C\e mk$ for a constant $C$.   In particular,  $N\le C\e mk$.  
\label{lm-cont1}\end{lemma}
\begin{proof}
Given $\e>0$ we can find $\e_0>0$ such that if $|\xi-\zeta|<\e_0$,  then   any convex representation 
$\xi=\sum_{z\in\range}\alpha_z z$ of $\xi$ satisfies 
  $\sum_{z\not\in\range_0}\alpha_z\le \e$.   (Otherwise we can let $\xi\to\zeta$
  and in the limit $\zeta$ would possess a convex representation with 
positive weight on $\range\smallsetminus\range_0$.) 
Consequently, if $x_{0,mk+\ell}\in\Pi_{mk,mk\xi}$ and $|\xi-\zeta|<\e_0$ the number of $(\range\smallsetminus\range_0)$-steps in $x_{0,mk+\ell}$ is bounded by $\e mk+\ell$.

Hence  it is enough to show that in each bad segment,
the number of $\range_0$-steps is at most a constant multiple of 
$(\range\smallsetminus\range_0)$-steps.  So consider a bad segment 
 $x_{s_i,s_i'}$.  If $s'_i=mk+\ell$ it can happen that 
$ 
 x_{s'_i}\cdot\uhat<\max_{s_i\le j\le s_i'}x_j\cdot\uhat.
$ 
In this case  we   add more steps  from $\range_0$ and increase  $s'_i$ so that 
 \begin{align}\label{complete}
 x_{s'_i}\cdot\uhat=\max_{s_i\le j\le s_i'}x_j\cdot\uhat.
  \end{align}
 This only makes things worse by increasing the number of $\range_0$-steps.  
 We proceed now by assuming \eqref{complete}. 
  
\begin{figure}[h]
\begin{center}
\includegraphics[height=5cm]{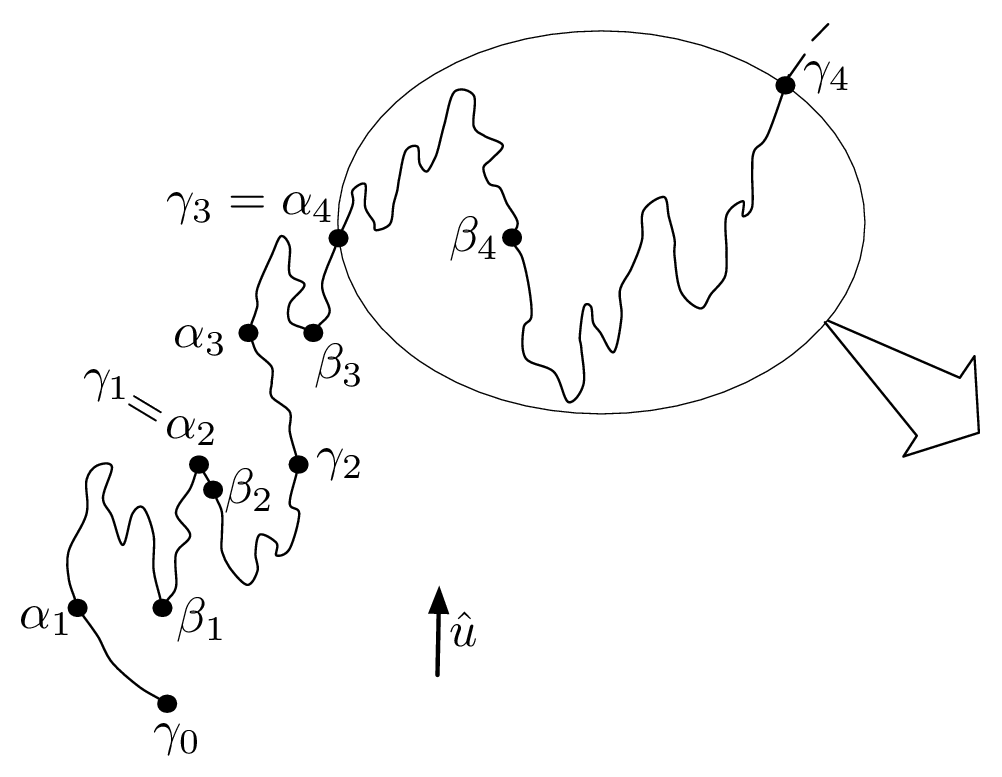}\ 
\includegraphics[height=3cm]{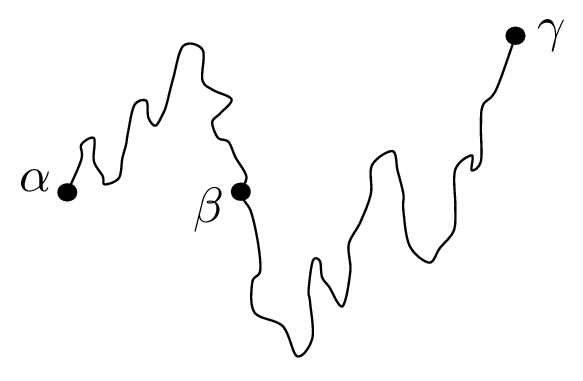}
\end{center}
\caption{Illustration of the   stopping times $\alpha_i$, $\beta_i$, and $\gamma_i$. Note how the immediate backtracking at $\gamma_1$ makes $\alpha_2=\gamma_1$ and $\beta_2=\alpha_2+1$.}\label{fig2}
\end{figure}

  Start with $\gamma_0=s_i$. Let 
\[ \alpha_1=\,  s_i'\,\wedge\, \inf\{ n\ge\gamma_0:   \exists j>n\text{ such that }x_j\cdot\uhat\le x_n\cdot\uhat\}. \]
We first control the number of $\range_0$-steps in the segment $z_{\gamma_0+1,\alpha_1}$.
The segment  $z_{\gamma_0+1,\alpha_1-1}$ cannot  
contain more than $\ell-1$  
$\range_0$-steps in a row because any $\ell$-string of $\range_0$-steps
 would have begun the next
good segment. Thus, the number of $\range_0$-steps in $z_{\gamma_0+1,\alpha_1}$ is bounded by 
 $(\ell-1)$ $\times$ (the number of $(\range\smallsetminus\range_0)$-steps) $+$ $\ell$. 
 Suppose 
  $\alpha_1=s_i'$, in other words, we already exhausted the entire bad segment.  Since a bad segment contains   at least one $(\range\smallsetminus\range_0)$-step 
  we are done:
the number of $\range_0$-steps is bounded by $2\ell$ times the number of $(\range\smallsetminus\range_0)$-steps.  
So let us  suppose $\alpha_1<s_i'$ and continue with the segment $x_{\alpha_1,s_i'}$.

Let 
\[ \beta_1=\inf\{ n>\alpha_1:  x_n\cdot\uhat\le x_{\alpha_1}\cdot\uhat\}\le s_i'  \]
be the time of the first backtrack after $\alpha_1$ and 
\[ \gamma_1=\inf\{ n>\beta_1:  x_n\cdot\uhat\ge  \max_{\alpha_1\le j\le \beta_1}
x_j\cdot\uhat\}  \]
  the time when the path gets at or above the previous maximum.  Due to \eqref{complete}, $\gamma_1\le s_i'$. 
 
 We claim that   
 in the segment $x_{\alpha_1,\gamma_1}$ the number of positive steps 
(in the $\uhat$-direction) is at most a constant times the number of nonpositive steps.
Since $\range_0$-steps are   positive steps while all nonpositive
steps are $(\range\smallsetminus\range_0)$-steps, this claim gives
the dominance
(number of $\range_0$-steps) $\le$    $C$ $\times$ (number of $(\range\smallsetminus\range_0)$-steps). 

 The claim is proved by counting.   
 Project all steps $z$ onto the $\uhat$ direction by considering 
 $z\cdot\uhat$, so that we can think of  
  a path on the 1 dimensional lattice.  Then, instead of the original steps that come in
  various sizes,  
count increments of  $\pm 1$.   Up to constant multiples, counting  unit increments is the same as counting  steps.  
 By the definition of the stopping times, at time $\beta_1$ the segment $x_{\alpha_1,\gamma_1}$  
visits a point at or below its starting level, 
but ends up at a new maximum level at time $\gamma_1$.  Ignore 
the part of the last step $z_{\gamma_1}$ that takes the path above the previous
maximum  $\max_{\alpha_1\le j\le \beta_1}
x_j\cdot\uhat$.    
Then each negative   unit increment in the $\uhat$-direction is matched by at most two positive unit increments.  (Project the right-hand picture in Figure \ref{fig2} onto
the vertical $\uhat$ direction.)  

Since the segment $x_{\alpha_1,\gamma_1}$ must have at least one $(\range\smallsetminus\range_0)$-step, we have shown that 
the number of $\range_0$-steps in the segment $x_{\gamma_0,\gamma_1}$ is bounded above by $2(C\vee\ell)$
$\times$  (number of $(\range\smallsetminus\range_0)$-steps).
Now repeat the previous argument, beginning at $\gamma_1$. Eventually
the bad segment  $x_{s_i,s_i'}$ is exhausted. 
\qed \end{proof}

Let $\bfv$ denote the collection of times $0=s_0\le s'_0<s_1<s'_1<s_2<s'_2<\dotsc<s_{N-1}<s'_{N-1}<s_{N}\le s'_{N}=mk+\ell$, 
positions $v_i=x_{s_i}$,  $v'_i=x_{s'_i}$, and the steps in bad path segments 
$u^{(i)}_{s_i,s'_i}=z_{s_i+1,s'_i}$.  $s_0=s'_0$ means $u^{(0)}$ is empty.

We use the following simple fact below.  Using Stirling's formula one can find a function $h(\e)\searrow 0$ 
such that, for all $\e>0$ and  $n\ge \e^{-1}$,  $\binom{n}{n\e}\le e^{nh(\e)}$. 

\begin{lemma}\label{lm:v-count}
With $\e>0$ fixed in Lemma \ref{lm-cont1}, and with $m$ large enough,
the number of vectors $\bfv$ is at most $C(mk)^{c_1}e^{mk h(\e)}$,
where the function $h$ satisfies   $h(\e)\to0$ as $\e\to0$.   
 \end{lemma}

\begin{proof}
Recall  $N\le C\e mk$ for a constant $C$ coming from Lemma \ref{lm-cont1}. 
We take $\e>0$ small enough so that $C\e<1/2$. 
 A vector $\bfv$ is determined by the following choices. 
 
 \medskip

(i)   The  times  $\{s_i, s'_i\}_{0\le i\le N}$   can be chosen in at most 
	\[\sum_{N=1}^{C\e mk}\binom{mk}{2N}\le Cmk\binom{mk}{C\e mk}\le C mk e^{mk h(\e)}\qquad\text{ways.}\]


(ii) The steps in the bad segments, in a total of at most
  $\abs{\range}^{C\e mk}\le e^{mk h(\e)}$ ways.

 \medskip

(iii)  The  path increments $\{v_{i}-v'_{i-1}\}_{1\le i\le N}$ across the good segments.
Their    number is also bounded by $C(mk)^{c_1}e^{mk h(\e)}$.  
 
\smallskip
 
The argument for (iii) is as follows. For each   finite $\range_0$-increment 
$y\in\{  z_1+\dotsm+z_k:  k\in\N, \, z_1,\dotsc, z_k\in \range_0\} $,   fix a particular 
representation $y=\sum_{z\in\range_0} a_z(y)z$, identified 
by the vector $a(y)= (a_z(y))\in\Z_+^{\range_0}$.  
The number of possible endpoints 
$\eta=\sum_{i=1}^N(v_{i}-v'_{i-1})$ is at most $C(\e mk)^d$ 
because $|mk\xi-mk\zeta|<mk\e$ and the total number of steps in all bad segments is at most $C\e mk$.
Each possible endpoint $\eta$ 
has at most $C(mk)^{\abs{\range_0}}$ representations  $\eta=\sum_{z\in\range_0} b_zz$
with $(b_z)\in\Z_+^{\range_0}$  because  projecting to $\uhat$  shows 
that each $b_z$ is bounded by $Cmk$.  Thus there are at most $C(mk)^{c_1}$  
vectors $(b_z)\in\Z_+^{\range_0}$ that can represent possible endpoints of the sequence
of increments.   Each such vector $b=(b_z)$ can be decomposed into a sum of 
increments   $b=\sum_{i=1}^{N} a^{(i)}$ in  at most 
\[ \prod_{z\in\range_0}\binom{b_z+N}{N} \le
{\binom{Cmk+C\e mk}{C\e mk}}^{\abs{\range_0}}
 \le e^{mk h(\e)} \] ways.   (Note that $\binom{a+b}{b}$ is increasing in both $a$ and $b$.)
So all in all  there are $C(mk)^{c_1} e^{mk h(\e)} $ possible sequences $\{a^{(i)}\}_{1\le i\le N}$
of increments in the space $\Z_+^{\range_0}$ that satisfy 
\[ \sum_{z\in\range_0} \sum_{i=1}^{N} a_z^{(i)}z =\eta  \qquad 
\text{for a possible endpoint $\eta$.}   \]

Map $\{v_{i}-v'_{i-1}\}_{1\le i\le N}$
  to  $\{a(v_{i}-v'_{i-1})\}_{1\le i\le N}$.   This mapping is 1-1.  The image is one
 of the previously counted sequences $\{a^{(i)}\}_{1\le i\le N}$ because
 \[  \sum_{z\in\range_0} \sum_{i=1}^{N} a_z(v_{i}-v'_{i-1}) z= 
 \sum_{i=1}^{N}  \sum_{z\in\range_0}a_z(v_{i}-v'_{i-1}) z=  
  \sum_{i=1}^{N} (v_{i}-v'_{i-1}) = \eta.  \]
We conclude that there are at most   $C(mk)^{c_1} e^{mk h(\e)} $  
sequences $\{v_{i}-v'_{i-1}\}_{1\le i\le N}$ of increments across the good segments.  Point (iii) has been verified. 
  
Multiplying counts (i)--(iii) proves the lemma.\qed
\end{proof}

Let $\Pi_{mk,mk\xi}^{\bfv}$ denote the paths in $\Pi_{mk,mk\xi}$ that are compatible
with $\bfv$, that is, paths 
that go through space-time points $(x_{s_i}, s_i)$, $(x_{s'_i}, s'_i)$ and take the
specified steps in the bad segments.  
The remaining unspecified   good segments   connect
$(x_{s'_{i-1}}, s'_{i-1})$ to $(x_{s_i}, s_i)$
 with   $\range_0$-steps, for $1\le i\le N$.  
 
Fix $\e>0$ small enough so that for large 
  $m$,    $C(mk)^{c_1} e^{mkh(\e)}\le e^{mk\delta}$.
Then our goal    \eqref{continuity-goal}   follows if we show 
\begin{align}\label{continuity-goal2}
\lim_{m\to\infty}\sum_{\bfv}\P\Big\{\sum_{x_{0,mk}\in\Pi_{mk,mk\xi}^\bfv}\!\!\!\! e^{mkR_{mk}^{\oneell}(g)}\ge e^{mk(\pres_\ell(g,\zeta)+\log|\range|)+5mk\delta}\Big\}=0.
\end{align}

\begin{figure}[h]
\begin{center}
\includegraphics[height=2.5cm]{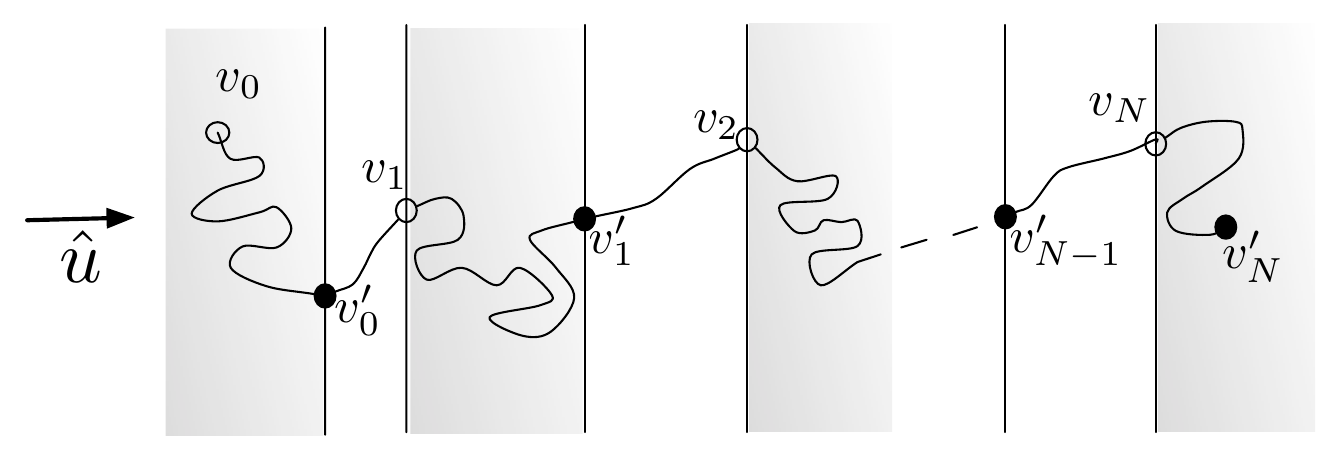}\ 
\includegraphics[height=2.5cm]{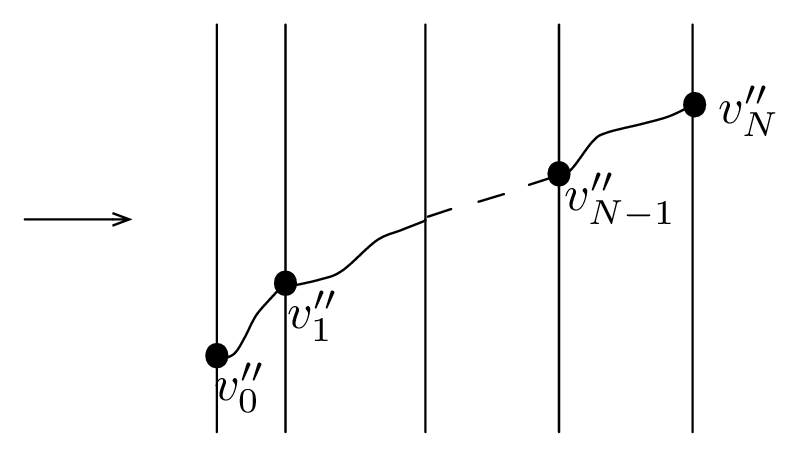}
\end{center}
\caption{Illustration of the construction. The shaded bad slabs of environments 
are deleted.  The white good slabs are joined together and shifted so that the good
path segments connect.  So for example points $v_1$ and $v'_1$ on the left
are identified as $v''_1$ on the right.  }\label{fig3}
\end{figure}

Given a vector $\bfv$ and an environment $\w$ define a new environment $\w^\bfv$
by deleting   the bad slabs and shifting the good slabs so that the good path
increments $\{v_{i}-v'_{i-1}\}_{1\le i\le N}$ become connected.  Here is a precise construction. 
First  for $x\cdot\uhat<0$ and  $x\cdot\uhat\ge\sum_{j=0}^{N-1}(v_{j+1}-v'_j)\cdot\uhat$ sample
 $\w^\bfv_x$  fresh (this part of space is irrelevant). 
  For a point $x$ in between  pick $i\ge 0$ such that 
\[\sum_{j=1}^{i}(v_{j}-v'_{j-1})\cdot\uhat\le x\cdot\uhat<\sum_{j=1}^{i+1}(v_{j}-v'_{j-1})\cdot\uhat\] 
and   put $y=\sum_{j=1}^{i}(v_j-v'_{j-1})$.  
Then set  $\w^\bfv_x=\w_{v'_i+x-y}$.  
	
For a fixed $\bfv$, each path $x_{0,mk+\ell}\in\Pi_{mk,mk\xi}^\bfv$ is mapped
in a 1-1 fashion to  a new path 
$x^{\bfv}_{0,\tau(\bfv)+\ell-1}$ as follows.  Set
\[\tau(\bfv)=\sum_{j=1}^{N}(s_{j}-s'_{j-1})-\ell.\] 
Given time point $t\in\{0,\dotsc, \tau(\bfv)+\ell-1\}$   pick  $i\ge0$ such that 
	\[\sum_{j=1}^{i}(s_{j}-s'_{j-1})\le t<\sum_{j=1}^{i+1}(s_{j}-s'_{j-1}).\]
Then with $s=\sum_{j=0}^{i}(s'_j-s_j)$ and  $u=\sum_{j=0}^{i}(v'_j-v_j)$
set   $x_t^{\bfv}=x_{t+s}-u$. 
This mapping of $\w$ and $x_{0,mk+\ell}$ moves the good slabs of environments
	together with the good path segments so that  $\w^\bfv_{x^{\bfv}_t}=\w_{x_{t+s}}$. 
	(See Figure \ref{fig3}.)  
 The sum of the good increments that appeared in Lemma \ref{lm:v-count} is now 
	\[x_{\tau(\bfv)+\ell}^{\bfv}=x_{s_{N}}-\sum_{j=0}^{N-1}(v'_j-v_j)=v_{N}-\sum_{j=0}^{N-1}(v'_j-v_j)=\sum_{j=1}^N(v_j-v'_{j-1}).\]
Define $\eta(\bfv)\in\Uset_0$ by
	\[x^\bfv_{\tau(\bfv)}=\tau(\bfv)\eta(\bfv).\]

Observe that $|\tau(\bfv)-mk|$ and $|x^{\bfv}_{\tau(\bfv)}-mk\xi|$ are
(essentially)  bounded by the total length of the bad segments and hence by $C\e mk$.
Moreover, due to total ergodicity $\pres_\ell(g,\cdot)$ is concave on $\Uset_0$ and hence continuous in its interior.
Thus, we can choose $\e>0$ small enough so that \[mk\pres_\ell(g,\zeta)+mk\delta>\tau(\bfv)\pres_\ell(g,\eta(\bfv)).\]
\eqref{continuity-goal2} would then follow if we show
\[\lim_{m\to\infty}\sum_{\bfv}\P\Big\{\sum_{x_{0,mk}\in\Pi_{mk,mk\xi}^\bfv} \!\!\!\!\!\!
e^{mkR_{mk}^{\oneell}(g)}\ge e^{\tau(\bfv)(\pres_\ell(g,\eta(\bfv))+\log|\range|)+3mk\delta}\Big\}=0.\]
This, in turn, follows from showing
\beq\begin{aligned}
\lim_{m\to\infty}\sum_{\bfv}\P\Big\{\sum_{x_{0,mk}\in\Pi_{mk,mk\xi}^\bfv} \!\!\!\!\!\!
&e^{\tau(\bfv)R_{\tau(\bfv)}^{\oneell}(g)(\w^{\bfv},x^{\bfv}_{0,\tau(\bfv)+\ell})}\\
&\quad\quad \ge \; e^{\tau(\bfv)(\pres_\ell(g,\eta(\bfv))+\log|\range|)+2mk\delta}\;\Big\}=0.
\end{aligned}\label{cont-17}
\eeq
To justify the step to \eqref{cont-17}, first   delete all terms from \[ mkR_{mk}^{\oneell}(g)=\sum_{i=0}^{mk-1}g(T_{x_i}\w, z_{i+1,i+\ell})\] that depend on $\w$ or $(z_i)$ outside of good slabs.  Since $g\le 0$ 
this  goes in the right direction.   The remaining terms can be written as
$\sum_i g(T_{x^{\bfv}_i}\w^{\bfv}, z^{\bfv}_{i+1,i+\ell})$ for a certain subset of indices
$i\in\{0,\dotsc, \tau(\bfv)-1\}$. 
Then  add in the terms for the remaining indices to capture the entire sum 
\[ \tau(\bfv)R_{\tau(\bfv)}^{\oneell}(g)(\w^{\bfv},x^{\bfv}_{0,\tau(\bfv)+\ell})
=\sum_{i=0}^{\tau(\bfv)-1} g(T_{x^{\bfv}_i}\w^{\bfv}, z^{\bfv}_{i+1,i+\ell}).\]  
The terms added correspond to  terms that originally straddled   good and bad segments.   
Hence since $g$ is local in its dependence on both $\w$ and $z_{1,\infty}$  there are at most $C\e mk$ such terms. 
Since $g$ is bounded, choosing $\e$ small enough allows us to absorb all such terms
into one $mk\delta$ error. 

Observing that $\w^\bfv$ has the same distribution as $\w$,    adding more paths
in the sum inside the probability,  and recalling that $|\tau(\bfv)-mk|\le Cmk\e$,
 we see that  it is enough to prove 
\[\lim_{m\to\infty}\sum_{\bfv}\P\Big\{\sum_{x_{0,\tau(\bfv)}\in\Pi_{\tau(\bfv),\tau(\bfv)\eta(\bfv)}} 
\!\!\!\!\!\!\!\!\!\!\!
e^{\tau(\bfv)R_{\tau(\bfv)}^{\oneell}(g)}\ge e^{\tau(\bfv)(\pres_\ell(g,\eta(\bfv))+\log|\range|)+\tau(\bfv)\delta}\Big\}=0.\]
By Lemma \ref{lm:v-count}, concentration inequality  Lemma \ref{lm:conc}, 
and $\tau(\bfv)\ge mk/2$,  the 
sum of  probabilities above  is bounded by 
$C(mk)^{c_1} e^{mkh(\e)-B\delta^2mk/2}\le C(mk)^{c_1} e^{-(\delta_1-h(\e))km}$
for another small positive constant $\delta_1$.  
Choosing $\e$ small enough shows  convergence to $0$ exponentially fast in $m$.  

We have verified the original goal \eqref{continuity-goal} and thereby 
completed the proof   of Theorem \ref{th:pt2pt cont}.\qed
\end{proof}
\end{theopargself}

\section{Quenched large deviations for the walk }\label{ld-sec}

Standing assumptions for this section are  Ê$\range\subset\Z^d$ is finite 
 and  $(\Omega,\kS,\P,\{T_z:z\in\gr\})$Ê
is a measurable ergodic dynamical system.   
The theorem below assumes
$\pres_\ell(g)$ finite; recall Remark \ref{finite-rem} for conditions that guarantee this.  
 We employ the following notation  for lower semicontinuous
regularization  of a function of several variables:
\[  F^{\text{lsc}(x)}(x,y)= \lim_{r\searrow 0} \inf_{z: \abs{z-x}<r}  F(z,y),  \]
and analogously for upper semicontinuous
regularization.

\begin{theorem}\label{th:ldp}
Let  $\ell\ge0$ and let  $g:\Omega\times\range^\ell\to\R$. 
Assume  $g\in\Ll$ and that $\pres_\ell(g)$ is finite. 
Then for $\P$-a.e.\ $\w$,  the distributions  $Q_n^{g,\w}\{X_n/n\in\cdot\}$ on $\R^d$ 
satisfy an LDP with deterministic rate function
 \beq  I^g(\zeta) =  \pres_\ell(g)-\pres^{\mathrm{usc}(\zeta)}_\ell(g,\zeta).  \label{I^g} \eeq
 This means that the following bounds  hold: 
 \begin{equation}
  \begin{aligned}
	&\varlimsup_{n\to\infty}n^{-1}\log  Q_{n}^{g,\w}\{X_n/n\in A\}
	\le-\inf_{\zeta\in A} I^g(\zeta)\ \text{ for  closed   }A\subset\R^d \\
\text{and}	\ 
 &\varliminf_{n\to\infty}n^{-1}\log Q_{n}^{g,\w}\{X_n/n\in O\}
	\ge-\inf_{\zeta\in O} I^g(\zeta) \ \text{ for   open   }O\subset\R^d.
\end{aligned}\label{qldp}
\end{equation} 
Rate function $I^g:\R^d\to[0,\infty]$  is convex, and  on $\Uset$   finite and  continuous.   
\end{theorem}

\begin{theopargself}
\begin{proof}[of Theorem {\ref{th:ldp}}]
  Let $O\subset\R^d$ be open, and $\zeta\in\Uset\cap O$. Then $\xhat_n(\zeta)\in nO$ 
for large $n$.  
	\begin{align*}
	&\varliminf_{n\to\infty}n^{-1}\log Q_{n}^{g,\w}\{X_n/n\in O\} \\
&\ge \varliminf_{n\to\infty} \biggl\{  n^{-1}\log E\bigl[e^{n R_n^{\oneell}(g)}\one\{X_n=\xhat_n(\zeta)\}\bigr]
-  n^{-1}\log E\bigl[e^{n R_n^{\oneell}(g)} \bigr] \,\biggr\}  \\
&= \pres_\ell(g,\zeta) - \pres_\ell(g).  
\end{align*}	
A supremum over an open set does not feel the difference between a function and its 
  upper semicontinuous regularization, and so we get the lower large deviation bound:
 \[   \varliminf_{n\to\infty}n^{-1}\log Q_{n}^{g,\w}\{X_n/n\in O\} \ge  - \inf_{\zeta\in O} 
 \{  \pres_\ell(g) - \pres^{\text{usc}}_\ell(g,\zeta) \}. \]
 
For a closed set $K\subset\R^d$ and $\delta>0$ Lemma \ref{lm-pressure2} implies
	\begin{align*}
	\varlimsup_{n\to\infty}n^{-1}\log Q_{n}^{g,\w}\{X_n/n\in K\}
	&\le-\lim_{\delta\searrow0}\inf_{\zeta\in K_\delta}\{\pres_\ell(g)-\pres_\ell(g,\zeta)\}\\
	&\le-\lim_{\delta\searrow0}\inf_{\zeta\in K_\delta}\{\pres_\ell(g)-\pres^{\mathrm{usc}}_\ell(g,\zeta)\}\\
	&=-\inf_{\zeta\in K}\{\pres_\ell(g)-\pres_\ell^{\mathrm{usc}}(g,\zeta)\}.
	\end{align*}
The last limit $\delta\searrow0$ follows from the compactness of $\Uset$. 
Properties of $I^g$ follow from Theorem \ref{th:regular}. 
 \qed
\end{proof} 
\end{theopargself}

\begin{remark}
Since the rate function $I^g$ is convex, it is the convex dual of the limiting logarithmic
moment generating function
\[  \sigma(t)= \lim_{n\to\infty} n^{-1}\log E^{Q_n^{g,\w}}(e^{t\cdot X_n}) = 
\pres_\ell(g+t\cdot z_1) - \pres_\ell(g)
\]
 on $\R^d$. This gives the identity 
\beq
-\pres_\ell^{\rm usc}(g,\zeta)=\sup_{t\in\R^d}\{\zeta\cdot t-\pres_\ell(g+t\cdot z_1)\}.
\label{pres19}\eeq
This identity can be combined with a variational representation for 
$\pres_\ell(g+t\cdot z_1)$ from Theorem 2.3 from \cite{Ras-Sep-Yil-12-} to 
produce a representation for $\pres_\ell^{\rm usc}(g,\zeta)$.  
\label{ldp-dual-rmk}\end{remark} 

As a corollary   we state a level 1 LDP for RWRE (see Example \ref{ex:rwre}).   

\begin{theorem} Let $d\ge 1$. Consider RWRE on $\Z^d$  in an ergodic environment 
with a finite set $\range\subset\Z^d$ of admissible steps.  
Assume that $g(\w,z)=\log p_{z}(\w)$ is a member of $\Ll$.
Then there exists
a continuous, convex rate function $I:\Uset\to[0,\infty)$  such that,  
  for $\P$-a.e.\ $\w$,  the distributions  $Q^{\w}\{X_n/n\in\cdot\,\}$ on $\Uset$ 
satisfy an LDP with rate $I$.  For $\zeta\in\ri\Uset$, $I(\zeta)$ is the limit of point probabilities:
\beq\label{level1 rate}
I(\zeta)=-\lim_{n\to\infty} n^{-1}\log Q^\w_0\{ X_n=\xhat_n(\zeta)\} \quad \text{a.s.}   
\eeq
\label{rwre-ldp|}\end{theorem}

%

This theorem complements our level 3 quenched LDPs  in \cite{Ras-Sep-11,Ras-Sep-Yil-12-} with formula \eqref{level1 rate} and the continuity of the rate function, in particular in the case where $0\not\in\Uset$ and $g$ is unbounded (e.g.\ if $\P$ has enough mixing 
and $g$ enough moments). To put the theorem in perspective we 
give a quick tour of the history of  quenched  large deviation theory of RWRE. 

The development began with the quenched level 1 LDP of 
Greven and den Hollander \cite{Gre-Hol-94} for 
  the one-dimensional elliptic nearest-neighbor i.i.d.\ case ($d=1$, $\range=\{-1,+1\}$, and $g$ bounded). Their proof utilized an auxiliary branching process.
The LDP was extended to the ergodic case by 
Comets, Gantert, and Zeitouni \cite{Com-Gan-Zei-00}, using hitting times.  Both results relied on the  possibility of explicit computations  in the one-dimensional nearest-neighbor case (which in particular implies $0\in\Uset$).
When $d\ge2$ Zerner \cite{Zer-98-aop}  used a subadditivity argument for certain passage times to prove the level 1 LDP 
in the nearest-neighbor  i.i.d.\ nestling case with  $g\in L^d$.   The nestling assumption ($0$ belongs to the convex hull of the support of $\sum_z zp_z(\w)$, and thus in particular $0\in\Uset$) was crucial for Zerner's 
argument.  Later, Varadhan \cite{Var-03-cpam}  used subadditivity directly to get the result for a general ergodic 
environment with finite step size, $0\in\Uset$, and bounded $g$. 

Subadditivity methods often fail to  provide   formulas for   rate functions. Rosenbluth \cite{Ros-06} used   the point of view of the particle,
following ideas of Kosygina, Rezakhanlou, and Varadhan \cite{Kos-Rez-Var-06} for 
diffusions with random drift, and gave an alternative proof of the quenched level 1 LDP along with two
variational formulas for the rate function. The assumptions were that the walk is nearest-neighbor,
$\P$ is ergodic, and $g\in L^p$ for  some $p>d$. That the walk is nearest-neighbor in \cite{Ros-06} is certainly not a serious
obstacle and can be replaced with a finite  $\range$  as long as $0\in\Uset$. 
Y\i lmaz \cite{Yil-09-cpam} extended the quenched LDP and rate function formulas to a univariate level 2 quenched LDP and Rassoul-Agha and Sepp\"al\"ainen \cite{Ras-Sep-11} extended further to level 3 results. 
%

All the past results mentioned above are for cases with $0\in\Uset$. This
restriction  eliminates natural  important models such as the space-time
case. 
When $0\not\in\Uset$, a crucial uniform integrability estimate fails  and the method of
 \cite{Kos-Rez-Var-06,Ros-06,Yil-09-cpam,Ras-Sep-11} breaks down.
For diffusions in time-dependent but bounded random potentials 
this issue was resolved by Kosygina and Varadhan \cite{Kos-Var-08}.  
 For  random polymers and  RWRE the way around this problem was found
 by 
  Rassoul-Agha, Sepp\"al\"ainen, and Y\i lmaz \cite{Ras-Sep-Yil-12-} who proved a quenched level 3 LDP with potential $g\in\Ll$ even when $0\not\in\Uset$. 
For the precise location of the difficulty see step 5 on page  833 of \cite{Kos-Var-08} and the proof of Lemma 2.13 of \cite{Ras-Sep-Yil-12-}. 
In a separate work \cite{Cam-etal-12-} we showed that 
the method  of \cite{Var-03-cpam} works 
also  in the space-time case  $\range\subset\{z:z\cdot e_1=1\}$, but with $g$ 
  assumed bounded.  

Limit \eqref{level1 rate} has been previously shown
for  various restricted cases: 
 in \cite{Gre-Hol-94} ($d=1$, $\P$ i.i.d., $\range=\{-1,1\}$, $g$ bounded), \cite{Zer-98-aop} ($\P$ i.i.d.\ , nestling, $g\in L^d$), 
\cite{Var-03-cpam} ($\P$ ergodic, $0\in\Uset$, $g$ bounded), and \cite{Cam-etal-12-} ($\P$ ergodic, $g$ bounded, and $\range\subset\{z:z\cdot e_1=1\}$).
\cite{Gre-Hol-94,Cam-etal-12-} also proved continuity of the rate function. 


Let us finally  point out that   \cite{Arm-Sou-12} obtains homogenization results similar to 
\cite{Kos-Var-08}   for unbounded potentials, but has   to compensate   with a mixing  assumption. This is the same spirit in which
our assumption  $g\in\Ll$ works.  

 

\section{Entropy representation of the point-to-point free energy} 
\label{iid-sec}

With either a compact $\Omega$ or an i.i.d.\ directed setting,  
the LDP of Theorem \ref{th:ldp} can  be obtained by contraction from
the higher level LDPs of \cite{Ras-Sep-Yil-12-}.  This is the route to linking
$\pres_\ell(g,\zeta)$ with entropy.   First we define the entropy.  

 The joint evolution of the environment and the walk give a Markov chain  
 $(T_{X_n}\w, Z_{n+1,n+\ell})$ on the state space $\bigom_\ell=\Omega\times\range^\ell$.
  Elements of $\bigom_\ell$ are 
denoted by  $\wz=(\w,\,z_{1,\ell})$.    The transition kernel is 
 	\begin{align}\label{pell-def}
		&\pell(\wz,\Sopr_z\wz)=\tfrac1{|\range|} \, \text{ for }z\in\range\text{ and }\wz=(\w,z_{1,\ell})\in\bigom_\ell    
	\end{align}
 where the transformations  $\Sopr_z$ are defined by  
 $\Sopr_z(\w,z_{1,\ell})=(T_{z_1}\w, (z_{2,\ell},z))$.   
An  entropy $\ratell$ that is naturally associated to this Markov chain and   reflects the
role of the background measure is defined as follows.  
   Let $\mu_0$ denote the $\Omega$-marginal of a probability measure $\mu\in\measures(\bigom_\ell)$.
  Define  
 	\begin{align}\label{Helldef}
		\ratell(\mu)=
			\begin{cases}
				\inf\{H(\mu\times q\,|\,\mu\times \pell):q\in\MC(\bigom_\ell)\text{ with }\mu q=\mu\}&\text{if }\mu_0\ll\P,\\
				\infty&\text{otherwise.}
			\end{cases}
	\end{align}
The infimum is over Markov kernels $q$ on $\bigom_\ell$ that fix $\mu$.
Inside the braces the familiar relative entropy is  
\beq
	H(\mu\times q\,|\,\mu\times\pell)
	= \int_{\bigom_\ell} \sum_{z\in\range}q(\wz,\Sopr_z\wz)\,\log\frac{q(\wz,\Sopr_z\wz)}{\pell(\wz,\Sopr_z\wz)}\,\mu(d\wz).
\eeq
Obviously  $q(\wz,\Sopr_z\wz)$ is not the most general Markov kernel on $\bigom_\ell$. 
But the entropy cannot be finite unless the kernel is supported on shifts  $\Sopr_z\wz$,
so we might as well restrict to this case.   
$\ratell: \measures(\bigom_\ell)\to [0,\infty]$ is convex. (The  argument for this
can be found at the end of Section 4 in \cite{Ras-Sep-11}.)   

 The quenched free energy has  this  variational characterization for $g\in\mathcal L$
   (Theorem 2.3 in  \cite{Ras-Sep-Yil-12-}):
 \beq  \pres_\ell(g)=\sup_{\substack{\mu\in\measures(\bigom_\ell), c>0}}
\bigl\{E^\mu[\min(g,c)]-\ratell(\mu)\bigr\}.   
\label{pr-ent-4.1}
\eeq
Our goal is to find such characterizations for the point-to-point free energy. 
We develop the formula in the i.i.d.\ directed setting.  Such a formula is also
valid in the more general setting of this paper if  $\Omega$ is a compact metric space. 
Details can be found in the preprint version \cite{Ras-Sep-12-arxiv}.


Let 
$\Omega=\Gamma^{\Z^d}$ be a product space with shifts $\{T_x\}$ 
and    $\P$  an i.i.d.\ product measure as in Example \ref{ex:product}.
Assume   $0\notin\Uset$.  Then 
the free energies $\pres_\ell(g)$ and  $\pres_\ell(g, \zeta)$
are deterministic (that is, the $\P$-a.s.\ limits are independent of the environment
$\w$) and $\pres_\ell(g, \zeta)$ is a continuous, concave function of $\zeta\in\Uset$.  
Assume also that $\Gamma$ is a separable  metric space,  and that   $\kS$  is   the product  
of Borel  $\sigma$-algebras, thereby also the  
Borel $\sigma$-algebra of $\Omega$.  

To utilize convex analysis we put  the space $\mathcal M$ of finite Borel measures
on $\bigom_\ell$ in duality with $C_b(\bigom_\ell)$, the space of bounded continuous
functions on $\bigom_\ell$, via integration:   $\langle f,\mu\rangle=\int f\,d\mu$. 
Give $\mathcal M$ the weak topology generated by $C_b(\bigom_\ell)$.  
Metrize $C_b(\bigom_\ell)$ with the supremum norm.   
The limit definition \eqref{xipres}  shows 
 that  $\pres_\ell(g)$ and $\pres_\ell(g,\zeta)$ are  Lipschitz in $g$,
uniformly in $\zeta$.  
$\ratell$ is extended to $\mathcal M$ by setting $\ratell(\mu)=\infty$ for 
measures $\mu$ that are not probability measures.  

For $g\in C_b(\bigom_\ell)$ \eqref{pr-ent-4.1} says that $\pres_\ell(g)=\ratell^*(g)$,
the convex conjugate of $\ratell$.  The double convex conjugate 
\beq \ratell^{**}(\mu) =  \pres^*_\ell(\mu)
= \sup_{f\in C_b(\bigom_\ell)}\{E^\mu[f]- \pres_\ell(f)\},  \quad\mu\in\measures(\bigom_\ell),
\label{H**}\eeq
is equal to  the  lower semicontinuous regularization $\ratell^{\rm{lsc}}$ of  $\ratell$    
(Propositions~3.3 and 4.1 in \cite{Eke-Tem-99} or Theorem~5.18 in \cite{Ras-Sep-10-ldp-}).
Since relative entropy is lower semicontinuous,  \eqref{Helldef} implies that 
\beq
\ratell^{**}(\mu)=\ratell(\mu)  \quad\text{ for $\mu\in\measures(\bigom_\ell)$
such that  $\mu_0\ll\P$. }
\label{H4}\eeq

There is a  quenched   LDP for
the distributions 
$Q_n^{g,\w}\{R_n^{\oneell}\in\cdot\}$, where $R_n^{\oneell}$
is    the empirical measure defined in \eqref{Rnell}.   
 The rate function of this LDP is    $\ratell^{**}$ 
(Theorems~3.1 and 3.3 of \cite{Ras-Sep-Yil-12-}). 

The  reader may be concerned about considering 
the $\P$-a.s. defined functionals  $\pres_\ell(g)$ or $\pres_\ell(g, \zeta)$
on the possibly non-separable function space  $C_b(\bigom_\ell)$.   However,  
for bounded functions we can integrate over the limits \eqref{xipres} and 
\eqref{pres-limit}  and define the free energies without any ``a.s.\ ambiguity'', so
for example
\[   \pres_\ell(g,\zeta)=
	\lim_{n\to\infty}n^{-1}  \E\Bigl( \log 
E\big[e^{n R_n^{\oneell}(g)}\one\{X_n=\xhat_n(\zeta)\}\big]\Bigr).   \]

We extend the  duality set-up to  involve  point to point free energy. 

\begin{theorem}   Let 
$\Omega=\Gamma^{\Z^d}$ be a product of separable  metric spaces  with 
Borel $\sigma$-algebra $\kS$,  shifts $\{T_x\}$, and an 
   an i.i.d.\ product measure  $\P$.  
 Assume   $0\notin\Uset$. 
With $\ell\ge 1$, let  $\mu\in\measures(\bigom_\ell)$ and 
$\zeta=E^\mu[Z_1]$.   Then 
\beq
\ratell^{**}(\mu)  = \sup_{g\in C_b(\bigom_\ell)}\{E^\mu[g]- \pres_\ell(g, \zeta)\}.
\label{H8}\eeq
On the other hand, for $f\in C_b(\bigom_\ell)$ and $\zeta\in\Uset$, 
\beq
\pres_\ell(f,\zeta)= 
\sup_{\mu\in\measures(\bigom_\ell) :\,E^\mu[Z_1]=\zeta}  \{E^\mu[f]- \ratell^{**}(\mu)   \}.  
\label{H9}\eeq 
Equation \eqref{H9} is valid also when $\ratell^{**}(\mu)$ is 
replaced with $\ratell(\mu)$:
\beq
\pres_\ell(f,\zeta)= 
\sup_{\mu\in\measures(\bigom_\ell) :\,E^\mu[Z_1]=\zeta}  \{E^\mu[f]- \ratell(\mu)   \}.  
\label{H9.1}\eeq 
\label{H-thm1}\end{theorem}

\begin{proof}
With fixed $\zeta$, introduce  the convex conjugate of $\pres_\ell(g,\zeta)$  by 
\beq \pres_\ell^*(\mu,\zeta)=\sup_{g\in C_b(\bigom_\ell)}\{E^\mu[g]-\pres_\ell(g,\zeta)\}.
\label{H10}\eeq  
Taking $g(\w,z_{1,\ell})=a\cdot z_1$ gives
$\pres_\ell^*(\mu,\zeta)\ge a\cdot(E^\mu[Z_1]-\zeta)-\log\abs{\range_0}.$
Thus  $\pres_\ell^*(\mu,\zeta)=\infty$ unless   $E^\mu[Z_1]=\zeta$.

From  Theorems \ref{th:regular} and \ref{th:pt2pt cont}, 
$E^\mu[g]-\pres_\ell(g,\zeta)$ is concave in $g$,
convex in $\zeta$, and continuous in both over $C_b(\bigom_\ell)\times\Uset$. 
Since $\Uset$ is compact we can apply 
a minimax theorem such as  K\"onig's theorem  \cite{Kas-94,Ras-Sep-10-ldp-}. 
Utilizing \eqref{sup pt2pt}, 
\begin{align*}
\pres_\ell^*(\mu)&=\sup_{g\in C_b(\bigom_\ell)}\{E^\mu[g]-\pres_\ell(g)\}\\
&=\sup_{g\in C_b(\bigom_\ell)}\inf_{\zeta\in \Uset}\{E^\mu[g]-\pres_\ell(g,\zeta)\}
=\inf_{\zeta\in \Uset}\pres_\ell^*(\mu,\zeta).
\end{align*}
Thus, if $E^\mu[Z_1]=\zeta$, then $\pres_\ell^*(\mu)=\pres_\ell^*(\mu,\zeta)$. 
Since $\ratell^{**}(\mu)=\pres_\ell^*(\mu)$,  \eqref{H8} follows from \eqref{H10}.  

By double convex duality  (Fenchel-Moreau theorem, see e.g.\ \cite{Ras-Sep-10-ldp-}),  
  for $f\in C_b(\bigom_\ell)$,  
\[\pres_\ell(f,\zeta)=\sup_\mu\{E^\mu[f]-\pres_\ell^*(\mu,\zeta)\}=
\sup_{\mu :\,E^\mu[Z_1]=\zeta}  \{E^\mu[f]-\pres_\ell^*(\mu)\}        \]
and \eqref{H9} follows. 

To replace $\ratell^{**}(\mu)$ with $\ratell(\mu)$ in \eqref{H9},  first consider the case 
  $\zeta\in\ri\Uset$.
\begin{align*}
&\sup_{\mu\in\measures(\bigom_\ell) :\,E^\mu[Z_1]=\zeta}  \{E^\mu[f]- \ratell^{**}(\mu)  \}\\
&\qquad = \sup_{\mu\in\measures(\bigom_\ell) :\,E^\mu[Z_1]=\zeta}  \{E^\mu[f]- \ratell(\mu) \}^{\text{usc($\mu$)}}\\
 &\qquad = \Bigl( \; \, \sup_{\mu\in\measures(\bigom_\ell) :\,E^\mu[Z_1]=\zeta}  \{E^\mu[f]- \ratell(\mu) \}\Bigr)^{\text{usc($\zeta$)}}\\
&\qquad  = \sup_{\mu\in\measures(\bigom_\ell) :\,E^\mu[Z_1]=\zeta}  \{E^\mu[f]- \ratell(\mu) \}. 
\end{align*}
The first equality is the continuity of $\mu\mapsto E^\mu[f]$.  
  The second is a consequence of 
the compact sublevel sets of $\{\mu: \ratell^{**}(\mu)\le c\}$.  This
compactness follows from  the exponential
tightness in the LDP controlled by the rate  $\ratell^{**}$, given by Theorem 3.3 in \cite{Ras-Sep-Yil-12-}.
The last equality follows because concavity gives continuity on $\ri\Uset$.  

For $\zeta\in\Uset\smallsetminus\ri\Uset$, let 
    $\Uset_0$ be the unique face such that $\zeta\in\ri\Uset_0$.  Then 
$\Uset_0=\conv\range_0$ where $\range_0=\Uset_0\cap\range$, and any path
to $\xhat_n(\zeta)$ will use only $\range_0$-steps.   This case 
reduces to the one   already proved, because all the quantities in \eqref{H9.1} 
are the same as those in a  new model where $\range$ is replaced by $\range_0$
and then $\Uset$ is replaced by $\Uset_0$. 
(Except for the extra terms coming from renormalizing the restricted
 jump kernel $\{\hat p_z\}_{z\in\range_0}$.) 
In particular,   $E^\mu[Z_1]=\zeta$ forces $\mu$ to be supported on
$\Omega\times\range_0^\ell$, and consequently any kernel $q(\wz, S_z^+\wz)$ 
that fixes $\mu$ is supported on shifts by  $z\in\range_0$.  
 \qed \end{proof}  

Next we extend the duality to certain  $L^p$ functions.  

\begin{corollary}  Same assumptions on $\Omega$, $\P$ and $\range$ as in Theorem
\ref{H-thm1}.
 Let $\mu\in\measures(\bigom_\ell)$ and $\zeta=E^\mu[Z_1]$. 
  Then the inequalities
\beq    E^\mu[g]-\pres_\ell(g) \le  \ratell^{**}(\mu)
\label{dual3}\eeq
and 
\beq    E^\mu[g]-\pres_\ell(g, \zeta) \le  \ratell^{**}(\mu)
\label{dual3.1}\eeq
are valid for all functions $g$ such that $g(\cdot, z_{1,\ell})$ is local
and in $L^p(\P)$ 
for all $z_{1,\ell}$ and some $p>d$, and $g$ is either bounded above or bounded
below.  
\end{corollary}

\begin{proof}  Since $\pres_\ell(g, \zeta)\le \pres_\ell(g)$, \eqref{dual3} is
a consequence of \eqref{dual3.1}.  
  Let $\mathcal H$ denote the class of functions $g$ that 
satisfy   \eqref{dual3.1}.   $\mathcal H$
contains bounded   continuous local functions by \eqref{H8}.  
 
Bounded pointwise convergence implies $L^p$  convergence.  So by the
$L^p$ continuity of   $\pres_\ell(g, \zeta)$  
 (Lemma  \ref{lm:animal1}(b)),   $\mathcal H$  is closed under bounded 
 pointwise convergence of local functions with common support.   
 General principles now imply that $\mathcal H$ 
contains all bounded local Borel functions.  To reach the last generalization
 to functions bounded from only  one side,   
observe that their truncations converge both monotonically and in 
$L^p$, thereby making  both $E^\mu[g]$   and $\pres_\ell(g, \zeta)$ 
converge.    
\qed
\end{proof}    

Equation~\eqref{H9} gives us a variational representation for $\pres_\ell(g,\zeta)$  but
only for bounded continuous $g$.  We come finally to one of our main results,
the variational representation for general potentials $g$.  
 
\begin{theorem}\label{th:pr-ent2}
 Let 
$\Omega=\Gamma^{\Z^d}$ be a product of separable  metric spaces  with 
Borel $\sigma$-algebra $\kS$,  shifts $\{T_x\}$, and     an i.i.d.\ product measure  $\P$.  
 Assume   $0\notin\Uset$. 
Let $g:\bigom_\ell\to\R$ be a function such that  for each $z_{1,\ell}\in\range^\ell$, 
  $g(\cdot,z_{1,\ell}) $ is a local function of $\w$ and a member of $L^p(\P)$ for some $p>d$.  
Then  
for all  $\zeta\in\Uset$,  
\begin{align}\label{var-pt2pt2}
 \pres_\ell(g,\zeta)=\sup_{\substack{\mu\in\measures(\bigom_\ell): E^\mu[Z_1]=\zeta\\[2pt]c>0}}
\bigl\{E^\mu[g\wedge c]-\ratell^{**}(\mu)\bigr\}.
\end{align}
Equation  \eqref{var-pt2pt2} is valid also when $\ratell^{**}(\mu)$ is 
replaced with $\ratell(\mu)$.
\end{theorem}

\begin{theopargself}
\begin{proof}
  From \eqref{dual3.1}, 
\[  \pres_\ell(g,\zeta)\ge \pres_\ell(g\wedge c,\zeta)\ge E^\mu[g\wedge c]-\ratell^{**}(\mu).\]
Supremum on the right over $c$ and $\mu$ gives 
\begin{align} \label{aux49}
\pres_\ell(g,\zeta)\ge \sup_{\substack{\mu\in\measures(\bigom_\ell): E^\mu[Z_1]=\zeta\\[2pt]c>0}}
\bigl\{E^\mu[\min(g,c)]-\ratell^{**}(\mu)\bigr\}.
\end{align}

For the other direction, let $c<\infty$ and  abbreviate $g^c=g\wedge c$.   
  Let $g_m\in C_b(\bigom_\ell)$ be a sequence converging to $g^c$ in $L^p(\P)$. 
  Let $\e>0$.  
By \eqref{H9} we can   find $\mu_m$ such that 
$E^{\mu_m}[Z_1]=\zeta$, $\ratell^{**}(\mu_m)<\infty$  and 
\beq \pres_\ell(g_m,\zeta)\le\e+E^{\mu_m}[g_m]-\ratell^{**}(\mu_m). \label{aux49.1}\eeq
 Take $\beta>0$  and  write
	\begin{align*}
	&\pres_\ell(g_m,\zeta)\le \e+E^{\mu_m}[g^c]-\ratell^{**}(\mu_m)+\beta^{-1}E^{\mu_m}[\beta(g_m-g^{c})]\\[4pt] 
	&\le\e+\sup\bigl\{E^\mu[g^c]-\ratell^{**}(\mu):c>0,\ E^\mu[Z_1]=\zeta\bigr\}\\[3pt]
	&\qquad+  \beta^{-1}\pres_\ell\bigl(\beta(g_m-g^{c})\bigr)     +\beta^{-1}\ratell^{**}(\mu_m)
	\\[4pt]
&\le\e+ \text{ [right-hand side of \eqref{var-pt2pt2}] } 	  \\[3pt] 
&\; + \varlimsup_{n\to\infty}\max_{x_k-x_{k-1}\in\range}\,n^{-1}\sum_{k=0}^{n-1}\abs{g_m(T_{x_k}\w, z_{1,\ell})-g^{c}(T_{x_k\w}, z_{1,\ell})}  +\beta^{-1}\ratell^{**}(\mu_m) \\[4pt]
&\le\e+ \text{ [right-hand side of \eqref{var-pt2pt2}] } 	  \\[3pt] 
 &\qquad+C\E \bigl[\; \max_{z_{1,\ell}\in\range^\ell}\abs{g_m-g^{c} }^p\,\bigr]
 +\beta^{-1}\ratell^{**}(\mu_m). 
	\end{align*}  
 The second inequality above used \eqref{dual3},  
 and the last  inequality  used   \eqref{animal1} and Chebyshev's inequality.    
Take first $\beta\to\infty$, then $m\to\infty$,  and last $c\nearrow\infty$ and  $\e\searrow0$.
Combined with \eqref{aux49}, we have 
 arrived at \eqref{var-pt2pt2}.  
 
Dropping $^{**}$   requires no extra work.  Since 
$\ratell\ge\ratell^{**}$, \eqref{aux49} comes for free.  For the complementary
inequality simply replace $\ratell^{**}(\mu_m)$   with $\ratell(\mu_m)$ in \eqref{aux49.1},
as justified by the last line of Theorem \ref{H-thm1}.  
  \qed\end{proof}\end{theopargself}
  


\section{Example: directed polymer in the $L^2$ regime} 
\label{L2-sec} 
We illustrate the variational formula of the previous section with 
  a directed polymer in the $L^2$ regime. 
The maximizing processes are basically the Markov chains constructed 
by Comets and Yoshida \cite{Com-Yos-06}  
and  Yilmaz \cite{Yil-09-aop}. 
   We restrict to  $\zeta\in\ri\Uset$. 
The closer $\zeta$ is to the relative boundary, 
the smaller we need to take   the inverse temperature $\beta$.

The setting is that of Example \ref{ex:dir-iid} with some
 simplifications.   $\Omega=\R^{\Z^{d+1}}$
is a product space indexed by the space-time lattice where $d$ is the
spatial dimension and the last coordinate direction is reserved for time.   
The environment is     $\w=(\w_x)_{x\in\Z^{d+1}}$ and 
 translations are $(T_x\w)_y=\w_{x+y}$.  
The coordinates $\w_x$ are i.i.d.\ under $\P$.    
The set of admissible steps is of the form $\range=\{(z', 1): z'\in\range'\}$ 
for a finite set $\range'\subset\Z^d$.  

To be in the {\sl weak disorder} regime 
we   assume that the difference of two $\range$-walks is 
at least $3$-dimensional.  
Precisely speaking, the additive subgroup
of $\Z^{d+1}$ generated by $\range-\range=\{x-y: x,y\in\range\}$ is
 linearly  isomorphic to some $\Z^m$, and we  
\beq
\text{assume that the dimension $m\ge 3$. } 
\label{irr-ass}\eeq
For example,  $d\ge 3$ and $\range'=\{\pm e_i: 1\le i\le d\}$ given by
simple random walk qualifies.  

The $P$-random walk has a kernel $(p_z)_{z\in\range}$.  Earlier we assumed
$p_z=\abs{\range}^{-1}$, but this is not necessary for the results, any 
fixed kernel will do.  We do assume $p_z>0$ for each $z\in\range$. 
  
The potential is 
$\beta g(\w_0,z)$ where $\beta\in(0,\infty)$ is the inverse temperature parameter.  Assume  
\beq  \E[e^{c\abs{g(\w,z)}}] <\infty \quad\text{ for some $c>0$ and all $z\in\range$.}    \label{st-ass1}\eeq
Now  $\pres_1(\beta g,\cdot\,)$ is well-defined and
continuous on $\Uset$. 

Define an averaged  logarithmic moment generating function
\[ \lambda(\beta,\theta)=\log\sum_{z\in\range} p_z\,\E[e^{\beta g(\w_0,z)+\theta\cdot z}] \quad\text{for $\beta\in[-c,c]$ and $\theta\in\R^{d+1}$.}\]
Under a fixed $\beta$, define the convex dual in the $\theta$-variable  by 
\beq \lambda^*(\beta,\zeta)=\sup_{\theta\in\R^{d+1}}\{\zeta\cdot\theta-\lambda(\beta,\theta)\} ,
\qquad \zeta\in\Uset.  \label{lambda*}\eeq
 For each $\beta\in[-c,c]$ and $\zeta\in\ri\Uset$ there exists $\theta\in\R^{d+1}$ such that 
$\nabla_\theta\lambda(\beta,\theta)=\zeta$ and this $\theta$ maximizes in \eqref{lambda*}.
A point $\eta\in\R^{d+1}$ also maximizes if and only if 
\beq
\text{$(\theta-\eta)\cdot z$ is constant over   $z\in\range$.}  
\label{max7}\eeq
Maximizers cannot be unique now because the last coordinate $\theta_{d+1}$ can 
vary freely without altering the expression in braces in \eqref{lambda*}.  
The spatial part $\theta'=(\theta_1,\dotsc,\theta_d)$ of  a maximizer is unique
 if and only if $\Uset$ has nonempty $d$-dimensional interior.

Extend the random walk distribution $P$ to a two-sided walk $(X_k)_{k\in\Z}$ 
that satisfies $X_0=0$ and $Z_i=X_i-X_{i-1}$ for all $i\in\Z$, where the 
steps $(Z_i)_{i\in\Z}$ are i.i.d.\ $(p_z)$-distributed.  
For $n\in\N$ define forward and backward partition functions 
\[ 
Z_n^+=E\bigl[ e^{\beta\sum_{k=0}^{n-1}g(\w_{X_k},Z_{k+1})+\theta\cdot X_n}]
\ \ \text{and}\ \ 
Z_n^-=E\bigl[ e^{\beta\sum_{k=-n}^{-1}g(\w_{X_k},Z_{k+1})-\theta\cdot X_{-n}}]
\]
and martingales 
$
W_n^\pm=e^{-n\lambda(\beta,\theta)} Z_n^\pm
$
 with
$ \E W_n^\pm = 1 .$ 

Suppose we have   the $L^1$ convergence 
$ 
W_n^\pm\to W_\infty^\pm 
$
 for some $(\beta,\theta)$.  
Then  $\E W_\infty^\pm = 1$, and  by Kolmogorov's 0-1 law 
   $\P(W_\infty^\pm>0)=1$.  Define a probability
measure $\mu^\theta_0$ on $\Omega$ 
by 
\[  \int_\Omega f(\w)\,\mu^\theta_0(d\w) =  \E[  W_\infty^-W_\infty^+ f]. \] 
  Define a stochastic kernel from
$\Omega$ to $\range$ by
\[
q^\theta_0(\w, z)= p_z e^{\beta g(\w_0,z)-\lambda(\beta,\theta)+\theta\cdot z}
\frac{W_\infty^+(T_z\w)}{W_\infty^+(\w)}. 
\]
Property $\sum_{z\in\range} q^\theta_0(\w, z)=1$ comes from (one of) the identities 
\beq
W_\infty^\pm =\sum_{z\in\range} 
p_z e^{\beta g(\w_{a^{(\pm)}},z)-\lambda(\beta,\theta)+\theta\cdot z}
W_\infty^\pm\circ T_{\pm z} \quad \text{$\P$-a.s.} 
\label{W1}\eeq 
where $a^{(+)}=0$ and $a^{(-)}=-z$.  These  
are inherited from the one-step Markov decomposition of $Z_n^\pm$. 
  For  $\ell\ge 0$,  on $\bigom_\ell$  
   define the  probability  
measure $\mu^\theta$   by 
\beq\mu^\theta(d\w, z_{1,\ell})= \mu^\theta_0(d\w) q(\w, z_1)q(T_{x_{1}}\w, z_2)\dotsm q(T_{x_{\ell-1}}\w, z_\ell) \label{mutheta}\eeq
   where $x_j=z_1+\dotsm+z_j$, and the stochastic kernel 
\beq q^\theta((\w, z_{1,\ell}), (T_{z_1}\w, z_{2,\ell}z))= q^\theta_0(T_{x_{\ell}}\w, z) .
\label{qtheta}\eeq   
  We think of $\beta$ fixed and $\theta$ varying and so include only $\theta$ 
in the notation of $\mu^\theta$ and $q^\theta$.   
Identities \eqref{W1}  can be used to  show that   $\mu^\theta$ is invariant 
under the  kernel $q^\theta$, or explicitly, for any bounded measurable test function $f$, 
\beq
\sum_{z_{1,\ell}, z} \int_\Omega \mu^\theta(d\w, z_{1,\ell})  q^\theta_0(T_{x_{\ell}}\w, z)
f(T_{z_1}\w, z_{2,\ell}z)  \;=\;  \int_{\bigom_\ell} f\,d\mu^\theta. 
\label{inv4}\eeq
  By Lemma 4.1 of \cite{Ras-Sep-11} the Markov chain
with transition $q^\theta$ started with $\mu^\theta$ is an ergodic process.
Let us call in general $(\mu,q)$ a measure-kernel pair  if $q$ is a Markov kernel
and $\mu$ is an invariant  probability
measure: $\mu q=\mu$. 

\begin{theorem}  
\label{dir-pol-thm}
Fix a compact subset $\Uset_1$ in the relative interior of $\Uset$.
Then there exists $\beta_0>0$ such that, for $\beta\in(0,\beta_0]$ and 
$\zeta\in\Uset_1$,  we can choose  $\theta\in\R^{d+1}$ such that 
  the following holds.  First $\nabla_\theta\lambda(\beta,\theta)=\zeta$ and $\theta$ is a maximizer
  in \eqref{lambda*}. 
The  martingales $W_n^\pm$ are uniformly integrable and the 
  pair  $(\mu^\theta, q^\theta)$ is well-defined by  \eqref{mutheta}--\eqref{qtheta}. 
  We have 
 \beq   \pres_1(\beta g, \zeta)=-\lambda^*(\beta,\zeta).  
\label{pres7}\eeq 
A measure-kernel pair $(\mu,q)$ on $\bigom_1$ such that $\mu_0\ll\P$  satisfies 
\beq
 \pres_1(\beta g, \zeta)= E^\mu[\beta g] - H(\mu\times q\vert \mu\times \hat p_1)
 \label{var7}\eeq
 if and only if $(\mu,q)=(\mu^\theta, q^\theta)$.  
\end{theorem}

\begin{remark} Note that even though  $\nabla_\theta\lambda(\beta,\theta)=\zeta$ does not
pick a unique $\theta$,  by \eqref{max7}  replacing $\theta$  by another maximizer
does not change the martingales 
$W_n^\pm$ or  the pair  $(\mu^\theta, q^\theta)$.  Thus $\zeta$
determines   $(\mu^\theta, q^\theta)$   uniquely.   
\end{remark}

We omit the proof of Theorem \ref{dir-pol-thm}.  Details appear in the preprint  \cite{Ras-Sep-12-arxiv}.

\appendix
\section{A convex analysis lemma}

\begin{lemma}  Let $\mathcal I$ be a finite subset of $\R^d$ and $\zeta\in\conv\mathcal I$.
Suppose $\zeta=\sum_{z\in\mathcal I}\beta_zz$ with each $\beta_z>0$ and 
$\sum_{z\in\mathcal I}\beta_z=1$.   Let $\xi_n\in\conv\mathcal I$ be a sequence 
such that $\xi_n\to\zeta$.  Then  
 there exist  coefficients $\alpha^{n}_z\ge 0$  such that $\sum_{z\in\mathcal I}\alpha^{n}_z=1$,  
$\xi_{n}=\sum_{z\in\mathcal I}\alpha^n_zz$   and
for each $z\in\mathcal I$, $\alpha^n_z\to \beta_z$ as $n\to\infty$. 

Furthermore,  assume $\mathcal I\subset\Q^d$ and $\xi_n\in\Q^d$.   Then the 
coefficients $\alpha^n_z$ can be taken rational.  
 \label{lm:co1}\end{lemma}
 
\begin{proof}
First we reduce the proof to the case where there exists a subset $\mathcal I_0\subset
\mathcal I$ such that $\mathcal I_0$ is affinely independent and generates   the same affine hull 
as $\mathcal I$, and $\xi_n\in\conv\mathcal I_0$ for all $n$.   To justify this reduction, note 
that there are finitely many such sets $\mathcal I_0$, and each $\xi_n$ must lie in the
convex hull of some $\mathcal I_0$ (Carath\'eodory's Theorem \cite[Theorem~17.1]{Roc-70}
applied to the affine hull of $\mathcal I$). 
All  but finitely many
of the $\xi_n$'s are contained in   subsequences that lie 
in a particular $\conv\mathcal I_0$. 
The coefficients of the finitely many remaining $\xi_n$'s are irrelevant for  the claim made
in the lemma.  

After the above reduction, the limit $\xi_n\to\zeta$ forces $\zeta\in\conv\mathcal I_0$.  
  The points $\tilde z\in\mathcal I\smallsetminus\mathcal I_0$ lie in the affine hull of 
$\mathcal I_0$ and hence have barycentric coordinates:  
\[  \gamma_{z,\tilde z}\in\R\,, \quad  \tilde z=\sum_{z\in\mathcal I_0}\gamma_{z,\tilde z}z\,, 
\quad  \sum_{z\in\mathcal I_0}\gamma_{z,\tilde z}=1   
\quad \text{ for } \  \tilde z\in\mathcal I\smallsetminus\mathcal I_0. \]
Consequently 
\beq\begin{aligned}
\zeta=\sum_{z\in\mathcal I}\beta_zz 
= \sum_{z\in\mathcal I_0} \Bigl( \beta_z 
+ \sum_{\tilde z\in\mathcal I\smallsetminus\mathcal I_0}\gamma_{z,\tilde z}\beta_{\tilde z}   \Bigr)z
\equiv  \sum_{z\in \mathcal I_0}\bar\beta_{z} z
\end{aligned}\label{aa1}\eeq
where the last identity defines the unique barycentric coordinates $\bar\beta_{z}$
 of $\zeta$ relative to $\mathcal I_0$.  Define the $\mathcal I_0\times \mathcal I$
 matrix $A=
\bigl[ \;  I  \; \big\vert \; \{\gamma_{z,\tilde z} \} \;\bigr]$   where $I$ is the 
$\mathcal I_0\times \mathcal I_0$ identity matrix and $(z, \tilde z)$ ranges over
$\mathcal I_0\times (\mathcal I\smallsetminus\mathcal I_0)$.   Then 
\eqref{aa1} is the identity  $A\beta=\bar\beta$ for the (column)  vectors
$\beta=(\beta_z)_{z\in\mathcal I}$ and $\bar\beta=(\bar\beta_z)_{z\in\mathcal I_0}$.  
Since $\eta=[\bar\beta \; 0]^t$  is also a   solution of $A\eta=\bar\beta$, we can
write $\beta= [\bar\beta \; 0]^t + y$ with $y\in\ker A$. 

Let $\xi_n= \sum_{z\in \mathcal I_0}\bar\alpha^n_{z} z$ define the barycentric
coordinates of  $\xi_n$.  Since the coordinates are unique, 
$\xi_n\to\zeta$ forces $\bar\alpha^n\to\bar\beta$. 
Let $\alpha^n= [ \bar\alpha^n \; 0]^t + y$.   Then
$A\alpha^n=\bar\alpha^n$ which says that $\xi_{n}=\sum_{z\in\mathcal I}\alpha^n_zz$. 
Also $\alpha^n\to\beta$.   
Since $\beta_z>0$, inequality  $\alpha^n_z\ge 0$ fails at most  finitely many times, 
and  for finitely many $\xi_n$ we can replace the $\alpha^n_z$'s with any  
  coefficients  that exist by $\xi_n\in\conv\mathcal I$.  Lastly, for 
$\sum_{z\in\mathcal I} \alpha^n_z=1$ we need $\sum_{z\in\mathcal I}y_z=0$.  
This comes from $Ay=0$ because the column sums of $A$ are all $1$. 
This completes the proof of the first part of the lemma. 

Assume now that  $\mathcal I\subset\Q^d$ and $\xi_n\in\Q^d$. Then by Lemma 
A.1.\ in \cite{Ras-Sep-Yil-12-} the vector  $\bar\alpha^n$ is rational.  By Lemma A.2.\ in 
\cite{Ras-Sep-07}
we can find rational vectors $y^n\in\kernel A$ such that  $y^n\to y$.  This time
take  $\alpha^n= [ \bar\alpha^n \; 0]^t + y^n$. 
\qed\end{proof}

\section{A concentration inequality}

We state a concentration inequality for the case of a bounded potential $g$.  It comes
from the ideas of Liu and Watbled \cite{Liu-Wat-09}, in the form given by
Comets and Yoshida \cite{Com-Yos-11-}.   

\begin{lemma}   Let $\P$ be an i.i.d.\ product measure on a product space $\Omega=\Gamma^{\Z^d}$ with generic elements $\w=(\w_x)_{x\in\Z^d}$. 
Let $g:\bigom_\ell\to\R$ be a bounded measurable  function such that,  for each $z_{1,\ell}\in\range^\ell$, 
  $g(\cdot\,,z_{1,\ell}) $ is a local function of $\w$.   
 Let  
$\zeta\in\Uset$ and 
\beq
F_n(\w)=\log E\big[e^{\sum_{k=0}^{n-1}g(T_{X_k}\w,Z_{k+1,k+\ell})}\one\{X_n= \xhat_n(\zeta)\}\big]. 
\label{Fn}\eeq
 Let $\Uset_0$ be a face of $\Uset$ such that $\zeta\in\Uset_0$,
 and assume that  $0\not\in\Uset_0$. 

Then there exist   constants $B, c\in(0,\infty)$ such that, for all $n\in\N$ and  $\e\in(0, c)$,  
\beq
\P\{ \w:  \abs{F_n(\w)-n\pres_\ell(g,\zeta)} \ge n\e\}  \le 2e^{-B\e^2 n}. 
\label{Fn2}\eeq
\label{lm:conc}\end{lemma}

\begin{proof}  Since $n^{-1}\E F_n\to \pres_\ell(g,\zeta)$, we can prove instead
\beq
\P\{ \w:  \abs{F_n(\w)- \E F_n } \ge n\e\}  \le 2e^{-B\e^2 n}. 
\label{Fn3}\eeq
As before, with $\range_0=\range\cap\Uset_0$ we have $\Uset_0=\conv\range_0$, 
any admissible path $x_{0,n}$ with $x_n=\xhat_n(\zeta)$ uses only $\range_0$-steps,
and from $0\not\in\Uset_0$ follows the existence of $\uhat\in\Z^d$ such that 
$\uhat\cdot z\ge 1$ for all $z\in\range_0$.   
Set   $M_0=\max_{z\in\range_0} \uhat\cdot z$. 

Fix $r_0\in\N$ so that $g(\w, z_{1,\ell})$ depends on $\w$ only through
$\{\w_x:  \abs{x\cdot\uhat}<r_0\}$.  Let  $n_0\in\N$ be such  that 
$n_0r_0\ge M_0n+r_0$.  
On $\Omega$ define the filtration  $\cH_0=\{\emptyset, \Omega\}$,
 $\cH_j=\sigma\{\w_x:  x\cdot\uhat\le j r_0\}$ for $1\le j \le n_0$. 
Since $x_n\cdot\uhat\le M_0n$,  $F_n$ is  $\cH_{n_0}$-measurable.   

To apply Lemma A.1 of \cite{Com-Yos-11-}  we need to find $G_1,\dotsc, G_{n_0}\in L^1(\P)$ 
such that  
\beq  \E[G_j\vert\cH_{j-1}]=\E[G_j\vert\cH_{j}]  \label{Fn7}\eeq
and   
\beq   \E[e^{t\abs{F_n-G_j}} \,\vert\, \cH_{j-1}] \le  b   \label{Fn8}\eeq
for constants  $t,b \in(0,\infty)$ and all $1\le j\le n_0$.  

For the background random walk define stopping times 
\[  \rho_j=\inf\{k\ge 0:  x_k\cdot\uhat \ge (j-2)r_0\}  \]
  and  
\[ \sigma_j=\inf\{k\ge 0:  x_k\cdot\uhat \ge (j+1)r_0\} . \]
Abbreviate  $\varphi(x)=\one\{x= \xhat_n(\zeta)\}$.  
For $1\le j\le n_0$  put
\[  W_j= \exp\Bigl\{  \sum_{\substack{k:\, 0\le k< n\wedge\rho_j \\ \ \ \; n\wedge\sigma_j\le k<n }}  g(T_{x_k}\w,z_{k+1,k+\ell}) \Bigr\} \]
and 
\[   G_j(\w)=\log E[ W_j\,\varphi(X_n) ]. \]
Then $W_j$ does not depend on $\{\w_x:  (j-1)r_0\le {x\cdot\uhat}\le jr_0\}$ and
consequently \eqref{Fn7} holds by the independence of the $\{\w_x\}$. 

Let $t\in\R\smallsetminus(0,1)$.   By Jensen's inequality, 
\begin{align*}
e^{t(F_n-G_j)}\; &=\; \biggl( \frac{E[ W_j\,e^{\sum_{k=n\wedge\rho_j}^{n\wedge\sigma_j-1}g(T_{X_k}\w,Z_{k+1,k+\ell})} \,\varphi(X_n) ]}{E[ W_j\,\varphi(X_n) ]} \biggr)^t \\
&\le  \;  \frac{E[ W_j\,e^{t\sum_{k=n\wedge\rho_j}^{n\wedge\sigma_j-1}g(T_{X_k}\w,Z_{k+1,k+\ell})} \,\varphi(X_n) ]}{E[ W_j\,\varphi(X_n) ]} \\
&\le \; \frac{E[ W_j\,e^{C\abs{t} (\sigma_j-\rho_j)} \,\varphi(X_n) ]}{E[ W_j\,\varphi(X_n) ]}
\; \le \; e^{C\abs{t}} 
\end{align*} 
since $g$ is bounded and  $\sigma_j-\rho_j\le 3r_0$.   This implies
\eqref{Fn8} since $t$ can be taken of either sign.  

Lemma A.1 of \cite{Com-Yos-11-} now gives \eqref{Fn2}.   Note that parameter $n$
in  Lemma A.1 of \cite{Com-Yos-11-}  is actually our $n_0$.  But the ratio $n/n_0$ is bounded
and bounded away from zero so this discrepancy does not harm \eqref{Fn3}. 
\qed\end{proof}

\begin{acknowledgements}
F.\ Rassoul-Agha's work was partially supported by NSF Grant DMS-0747758. T.\ Sepp\"al\"ainen's work was partially supported
by NSF Grant DMS-1003651 and by the Wisconsin Alumni Research Foundation. The authors thank the anonymous referee for valuable comments that improved the
presentation. 
\end{acknowledgements}

\bibliographystyle{spmpsci}      
\bibliography{firasbib2010}   

\end{document}